\documentclass[12pt]{amsart}

\usepackage{amscd,amssymb,amsfonts,amsmath,amstext,amsthm,hyperref,bbm}
\usepackage[autolanguage]{numprint}
\usepackage[square,numbers,sort]{natbib}
\usepackage{doi}
\usepackage{mathscinet}

\def\clap#1{\hbox to 0pt{\hss#1\hss}}

\newcommand{\ds}{\displaystyle}

\def\BN{{\mathbb N}}

\def\BZ{{\mathbb Z}}

\def\D{{\mathcal D}}

\def\iff{\Leftrightarrow}

\newcommand{\cyc}[1]{\langle #1 \rangle}

\newcommand{\du}[1]{\mathbbm{k}^{#1}}

\DeclareMathOperator{\image}{Im}  
\DeclareMathOperator{\Img}{\image}
\DeclareMathOperator{\Hom}{Hom}   
\DeclareMathOperator{\id}{id}     
\DeclareMathOperator{\sgn}{sgn}   
\DeclareMathOperator{\class}{class}
\DeclareMathOperator{\Aut}{Aut}   
\DeclareMathOperator{\Inn}{Inn}   
\DeclareMathOperator{\End}{End}   
\DeclareMathOperator{\Char}{char} 
\DeclareMathOperator{\Rep}{Rep}		

\theoremstyle{plain}
\newtheorem{thm}{Theorem}[section]
\newtheorem*{theorem}{Theorem}
\newtheorem{cor}[thm]{Corollary}
\newtheorem{prop}[thm]{Proposition}
\newtheorem{lem}[thm]{Lemma}
\newtheorem{conj}[thm]{Conjecture}

\theoremstyle{definition}
\newtheorem{df}[thm]{Definition}
\newtheorem{example}[thm]{Example}

\theoremstyle{remark}
\newtheorem*{rem}{Remark}

\numberwithin{equation}{section}

\def\morphquad{(p,u,r,v)}                            
\def\flipquad{(p^*,v^*,r^*,u^*)}										 
\newcommand{\BCh}[1]{\Hom(#1,\widehat{#1})}
\def\BCG{\Aut(G)\ltimes\BCh{G}}                      
\DeclareMathOperator{\SpAutc}{SpAut_c}               
\def\BCGG{\SpAutc(G)\ltimes\BCh{G}}                  
\newcommand{\morph}[4]{\left( \begin{array}{cc} #1 & #2 \\ #3 & #4 \end{array} \right)}

\title{Automorphisms of the doubles of purely non-abelian finite groups}
\author{Marc Keilberg}
\address{Institut de Math\'{e}matiques de Bourgogne, Universit\'{e} de Bourgogne,\\ Facult\'{e} des Sciences Mirande, 9 avenue Alain Savary, BP 47870 21078\\ Dijon Cedex, France}
\email{keilberg@usc.edu}

\begin{document}
\begin{abstract}
Using a recent classification of $\End(\D(G))$, we determine a number of properties for $\Aut(\D(G))$, where $\D(G)$ is the Drinfel'd double of a finite group $G$.  Furthermore, we completely describe $\Aut(\D(G))$ for all purely non-abelian finite groups $G$.  A description of the action of $\Aut(\D(G))$ on $\Rep(\D(G))$ is also given.  We are also able to produce a simple proof that $\D(G)\cong\D(H)$ if and only if $G\cong H$, for $G$ and $H$ finite groups.
\end{abstract}
\maketitle
\tableofcontents

\section*{Introduction}\label{sec:intro}
The aim of this paper is to investigate the structure of Hopf automorphisms of $\D(G)$, where $\D(G)$ is the Drinfeld double of a finite group $G$.  We will be able to give a complete description when $G$ is purely non-abelian, meaning that $G$ has no non-trivial abelian direct factors.  The original motivation for this came from \citep{Co}, where it was observed that the automorphisms of a semisimple Hopf algebra $H$ permute its irreducible characters in a way that preserves all higher Frobenius-Schur indicators \citep{KSZ2}. This provides the ability to reduce the number of indicator computations actually performed, provided one has a reasonably robust and efficient way of computing the automorphisms and their action on irreducible modules.  Precise computation of indicator values has remained a difficult problem, relying on either computers or especially well-behaved groups \citep{Co,GMN,GM,K,K2}.  Some advances in this direction were made in \citep{IMM}, which provides an integrality test, but it remains unclear how to predict integrality with complete certainty, or how to predict negative indicators--in particular modules $V\cong V^*$ with a skew-symmetric form \citep{LM}--even when integrality is assured.

Automorphisms of Hopf algebras and quantum groups have been studied in several other works, including \citep{ABM,ACh,AD,AnDu}.  \citet{Rad90} showed that $\Aut(H)$ is finite over a field of characteristic zero when $H$ is semisimple, as well as in finite characteristic with certain additional assumptions.  Recently, \citet{SV} used $\Aut(H)$ to give a new, extended notion of higher Frobenius-Schur indicators \cite{KSZ2}, and applied their methods to $H_8$.  In \citep{ABM}, a classification of bicrossed products of Hopf algebras and their morphisms is provided.  The automorphisms of the quantum groups $H_{4n,\omega}$ were completely determined by these methods, in particular.  This classification of morphisms plays a central role in our investigations of $\Aut(\D(G))$.

All groups will be finite, and we work over an algebraically closed field $\mathbbm{k}$ of characteristic zero, unless otherwise indicated.  We use $\delta_{x,y}$ as the Kronecker delta symbol.  All homomorphisms and automorphisms are of Hopf algebras or groups, as appropriate, unless otherwise noted.   We denote the $\mathbbm{k}$-linear dual of a group $G$ by $\du{G}$, with dual basis $\{e_g\}_{g\in G}$.  We denote the group-like elements of $\du{G}$ by $\widehat{G}$, which we identify with the usual group of irreducible linear characters.  We identify $\BCh{G}$ with the group of $\mathbbm{k}$-bilinear bicharacters $G\times G\to\mathbbm{k}^\times$ in the usual way.  We denote all identity elements and trivial subgroups by 1 when there can be no confusion, with the exception of $\widehat{G}$ and $\du{G}$, whose identities we write as $\varepsilon$. When $K$ is a normal subgroup of $G$, we write $K\trianglelefteq G$.  We denote the conjugation actions of $G$ on itself and its dual both by $\rightharpoonup$: $x\rightharpoonup e_g = e_{xgx^{-1}}$ and $x\rightharpoonup g = xgx^{-1} = g^x$.  Given any group homomorphism $f$, we denote its linear dual by $f^*$.

The assumptions on the base field are made to assure that $\widehat{A}\cong A$ for any abelian group $A$.  This can be relaxed to any field where this isomorphism holds for all abelian subgroups of $G$, and in principle our results can be modified to work in an arbitrary field, provided one takes extra care concerning the potential failure of these isomorphisms. 

The paper is organized as follows.  In Section \ref{sec:pre} we cover all of the preliminary results and definitions needed, in particular giving a quick review of the main result we need from \citep{ABM}.  This result describes morphism $\D(G)\to\D(H)$ via a quadruplet of morphisms $\morphquad$, which can be equivalently interpreted as a description in terms of $2\times 2$ matrices of morphisms satisfying certain compatibility relations.  In Section \ref{sec:pu} we establish many important properties of $p$ and $u$ that will be used throughout the paper.  In section \ref{sec:flip} we introduce a special kind of homomorphism $\D(G)\to\D(H)$, and show that all elements of $\Aut(\D(G))$ are of this type.  This establishes a certain symmetry between the morphisms $u$ and $v$, providing easy proofs for some key results.  In Section \ref{sec:subgroups} we introduce a number of useful subgroups of $\Aut(\D(G))$.  Section \ref{sec:main} contains the main result, which can be stated as:
\begin{theorem}
	Let $G$ be a finite group.  Then $\Aut(\D(G))$ has subgroups $S,T,U,V$ with $S\cong \End(Z(G))$, $T\cong \BCh{G}$, $U\cong \Aut_c(G)=C_{\Aut(G)}(\Inn(G))$, and $V\cong\Aut(G)$, all of which intersect each other trivially.  Moreover, $\Aut(\D(G))$ is generated by these subgroups if and only if $G$ is purely non-abelian.
	
	As a special case, if $Z(G)=1$ then $\Aut(\D(G))\cong\BCG$, where the action is $r^v = v^* r v$.
\end{theorem}
In the last part, $v\in\Aut(G)$ gives the automorphism \[e_g\# h\mapsto e_{v(g)}\# v(h),\] as noted by \citet{Co}; and $r\in\BCh{G}$ gives the automorphism \[e_g\# h\mapsto r(h)e_g\# h.\]  If we write $r(g)=\sum \omega(g,x)e_x$, with $\omega$ a bicharacter, then this is equivalent to $e_g\# h\mapsto \omega(h,g)e_g\# h$.  The Theorem can be viewed as analogous to the description of the automorphisms of the direct products of groups given by \citet{BCM}.  Indeed, it can be proven from that description, but we do not take that approach here since it is less direct.  In Section \ref{sec:restriction} we consider the restriction of automorphisms to the group-like elements and describe the kernel of this restriction.  In Section \ref{sec:products} we determine when certain subgroups are normal, as well as when they have complements, including the kernel of the restriction.  This provides a number of different ways to write $\Aut(\D(G))$ with exact factorizations or even semidirect products for certain types of groups. We conclude the paper in Section \ref{sec:modules} with a discussion on the basic properties of the induced action of a class of morphisms on the categories of representations.  This includes the action of the automorphisms of $\D(G)$ on the category $\Rep(\D(G))$.  This provides the necessary details on how to use automorphisms to reduce the number of simple modules to consider when computing all indicators for this category.

\section*{Acknowledgments}
The author would like to think S. Montgomery, G. Mason, and P. Schauenburg for their many conversations and suggestions.

\section{Preliminaries}\label{sec:pre}
Our reference for the theory of Hopf algebras is \citep{Mo}.  We use Sweedler summation notation for the comultiplication: $\Delta(h)=h_{(1)}\otimes h_{(2)}$.  For the wide ranging uses and constructions of the Drinfeld double, we also refer the reader to \citep{DPR}.  Here, we take the following definition of $\D(G)$ over a field $\mathbbm{k}$, where $G$ is a finite group.  As a coalgebra, $\D(G)$ is $(\du{G})^\text{cop}\otimes kG$ with the usual tensor product coalgebra structure.  We denote a simple tensor of $\D(G)$ by $e_g\# h$, some $g,h\in G$.  Letting $\varepsilon=\sum_{g\in G} e_g$, the identity of $\D(G)$ is $\varepsilon\# 1$.  We will also use $\varepsilon$ as the counit of $kG$.  The multiplication is given by having $G$ act on $\du{G}$ by conjugation:
\[(e_g\# g')(e_h\# h') = \delta_{g,g'\rightharpoonup h} e_g \# g'h'.\]  We denote the functional on $\du{G}$ given by evaluation at $g\in G$ by $\operatorname{ev}_g$.

We begin by stating the main result from \citep{ABM} that we need here.

\begin{thm}\citep[Corollaries 2.3-2.4]{ABM}\label{ABMCor2.3}
Let $G, H$ be finite groups.  Then there exists a bijection between the set of all morphisms of Hopf algebras $\psi\colon \D(G)\to \D(H)$ and the set of all quadruples $\morphquad$ where $u\colon \mathbbm{k}^{G\text{ cop}}\to \mathbbm{k}^{H\text{ cop}}$, $r\colon \mathbbm{k}G\to \mathbbm{k}^{H\text{ cop}}$ are unitary coalgebra maps, and $p\colon \mathbbm{k}^{G\text{ cop}}\to \mathbbm{k}H$, $v\colon:G\to H$ are morphisms of Hopf algebras satisfying the following compatibility conditions:
    \begin{align}
      u(a_{(1)})\otimes p(a_{(2)}) &= u(a_{(2)})\otimes p(a_{(1)})\label{eq:upcocomm}\\ 
      u(ab) &= u(a_{(1)})\left(p(a_{(2)})\rightharpoonup u(b)\right)\label{eq:usplit}\\
      r(hg) &= r(h)\left( v(h)\rightharpoonup r(g)\right)\label{eq:rsplit}\\
      r(h)\left( v(h)\rightharpoonup u(b)\right) &= u(h\rightharpoonup b_{(1)}) \left( p(h\rightharpoonup b_{(2)})\rightharpoonup r(h)\right)\label{eq:allrel}\\
      v(h) p(b) v(h)^{-1} &= p(h\rightharpoonup b)\label{eq:vprel}
    \end{align}
for all $a,b\in \mathbbm{k}^{G\text{ cop}}$ and $g,h\in G$.

Under the above bijection the morphism of Hopf algebras $\psi\colon \D(G)\to \D(H)$ corresponding to $\morphquad$ is given by:
    \begin{align}
      \psi(a\# g) = u(a_{(1)})\left( p(a_{(2)})\rightharpoonup r(g)\right) \# p(a_{(3)}) v(g)\label{eq:morphpsi-pre}
    \end{align}
for all $a\in \mathbbm{k}^{G\text{ cop}}$ and $g\in G$.
\end{thm}
\begin{proof}
  We give only the definitions of $u,r,p,v$ for clarity.  Full details of the proof can be found in \citep{ABM}.
  \begin{align}
    u(a) &= \left((\id\otimes\varepsilon_{H})\circ\psi\right)(a\# 1_G)\label{udef}\\
    p(a) &= \left((\operatorname{ev}_{1_H}\otimes\id)\circ\psi\right)(a\# 1_G)\label{pdef}\\
    r(g) &= \left((\id\otimes \varepsilon_{H})\circ\psi\right)(\varepsilon_G\# g)\label{rdef}\\
    v(g) &= \left((\operatorname{ev}_{1_H}\otimes\id)\circ\psi\right)(\varepsilon_G\# g)\label{vdef}.
  \end{align}
\end{proof}
	The first relation is the definition for $u,p$ to cocommute.  Since group algebras are cocommutative, we also have that $r,v$ cocommute trivially, and that $u,r$ trivially have commuting images.  Equation \eqref{eq:vprel} expresses the commutation relation between the images of $v$ and $p$.
	
It is convenient to think of $\psi=\morphquad$ as a matrix
	\begin{align}
		\label{matrixform}\morph{u}{r}{p}{v},
	\end{align}
	with evaluations performed on the right.  We may multiply two such matrices together, where, as is standard \cite{BCM,Bi08}, addition is the convolution product--and thus subtraction indicates the antipode--, and multiplication is composition.  It is a straightforward exercise to verify that the relations in the Theorem are precisely what is needed to make such a matrix a morphism of Hopf algebras, and for matrix multiplication to correspond to composition of morphisms.  Thus the Theorem can be interpreted as describing $\Hom(\D(G),\D(H))$ as matrices of morphisms satisfying the given relations.

This paper is concerned with improving and understanding the above description, especially for isomorphisms.  We can immediately show that there is more structure to $r$ than was originally noticed, resulting in several simplifications. 

\begin{cor}\label{rishopf}
The morphism $r$ is a morphism of Hopf algebras, and is thus uniquely determined by a group homomorphism $r\in\Hom(G,\widehat{H})$.  Equation \eqref{eq:rsplit} is then a consequence of this, and may be omitted.  Furthermore, equation \eqref{eq:allrel} simplifies to
\begin{align}
v(h)\rightharpoonup u(b) = u(h\rightharpoonup b)\label{eq:vurel}
\end{align}
and we may instead write
\begin{align}
  \psi(a\# g) &= u(a_{(1)})r(g) \# p(a_{(2)}) v(g), \label{eq:morphpsi}
\end{align}
for all $a\in\mathbbm{k}^{G\text{ cop}}$ and $g\in G$.
\end{cor}
\begin{proof}
    Since $r$ is a morphism of unitary coalgebras, it must map group-likes to group-likes, and so $r(g)\in \widehat{H}$ for all $g\in G$.  Then for any $y\in H$, $y\rightharpoonup r(g) = r(g)$, and so the $H$-action is trivial on the image of $r$.  Thus equation \eqref{eq:rsplit} simplifies to $r(gh)=r(g)r(h)$ for all $g,h\in G$.  Therefore $r$ is also an algebra, and thus Hopf, morphism.  Equation \eqref{eq:morphpsi} follows similarly.
\end{proof}

Since we now know that $p$, $r$ and $v$ are morphisms of Hopf algebras, the question naturally arises as to whether or not $u$ is also a morphism of Hopf algebras.  We suspect this is always true, but we prove it only for a special case that includes the automorphisms in Corollary \ref{thetahopf1}.  Whenever this is the case, we can identify $u^*$ with a group homomorphism $\alpha\colon H\to G$.  It can easily be seen that $u$ Hopf implies \eqref{eq:usplit}, leaving us with just three non-trivial compatibility conditions in this case.

The goal now will be to describe $\Aut(\D(G))$ for as many finite groups as possible.  The simplest case is, of course, when $G$ is abelian.  For then $\D(G)\cong \mathbbm{k}(G\times G)$, and so $\Aut(\D(G))\cong\Aut(G\times G)$, which can be computed by classical methods \cite{Shoda28}.  We will see in Corollary \ref{cor:abelequiv} and example \ref{ex:badabel} that, in this case, there are automorphisms $\morphquad$ where both $u,v$ are not isomorphisms.  Indeed, there are also automorphisms $\morphquad$ where changing $r$ can fail to yield an automorphism. Such behaviors complicate the description of $\Aut(\D(G))$.  As such, one suspects that abelian direct factors are the precise cause of such behavior in general, and that a description of $\Aut(\D(G))$ for purely non-abelian groups should be more readily attainable.  We will prove this in Theorem \ref{thm:pure-uv-isoms}; see also Corollary \ref{cor:abelequiv} and Theorem \ref{thm:pure-equiv}.

\section{Essential properties of \ensuremath{p} and \ensuremath{u}}\label{sec:pu}
To discern the properties of $\Aut(\D(G))$, we need to explore the consequences the compatibilities have on the exact form the components of a morphism can take.  We will begin our investigation with $p$, starting with the following simple and useful observation.

\begin{thm}\label{lambdagroups}
  Let $p\colon \mathbbm{k}^{G\text{ cop}}\to kH$ be a morphism of Hopf algebras.  Then $p$ is uniquely determined by a group isomorphism $\widehat{A}\overset{f}{\cong} B$, where $A,B$ are abelian subgroups of $G,H$ respectively.  In this case, $p$ is given by the obvious imbedding of $\mathbbm{k}^A=\mathbbm{k}\widehat{A}$ into $\mathbbm{k}^G$, and setting $p(e_g)=0$ whenever $g\not\in A$, and $p(\chi)=f(\chi)$ for $\chi\in\widehat{A}$.

  Furthermore, if $v\in\Hom(G,H)$ satisfies equation \eqref{eq:vprel}, then also $A\trianglelefteq G$, and $B$ is closed under conjugation by elements of $v(G)$ (ie., $B^{v(G)}=B$).  In addition, $v$ maps $C_G(A)$ to $C_H(B)$.  In particular, $C_H(v(A))\subseteq C_H(B)$, with equality whenever $v$ is an isomorphism.
\end{thm}
\begin{proof}
  Since $\mathbbm{k}^{G\text{ cop}}$ is a commutative Hopf algebra, we must have that $\Img(p)$ is a commutative Hopf sub-algebra of $kH$.  Therefore, there exists an abelian subgroup $B\subseteq H$ such that $\Img(p)=\mathbbm{k}B$, and we may decompose $p$ as the composition
  \[(\mathbbm{k}^G)^\text{cop}\rightarrow \mathbbm{k}B\hookrightarrow kH.\]
  Now the dual morphism $p^*\colon \mathbbm{k}^H\to kG^\text{op}$ is also a morphism of Hopf algebras, and $\mathbbm{k}G^\text{op}$ is naturally isomorphic to $\mathbbm{k}G$, so there is an abelian subgroup $A\subseteq G$ such that $\Img(p^*) = \mathbbm{k}A$.  We can thus decompose $p^*$ as the composition
  \[\mathbbm{k}^H\twoheadrightarrow \mathbbm{k}^B\rightarrow \mathbbm{k}A\hookrightarrow \mathbbm{k}G^\text{op}.\]
  Dualizing again we have the following decomposition of $p$:
  \[(\mathbbm{k}^G)^\text{cop}\twoheadrightarrow \mathbbm{k}^A \cong \mathbbm{k}A \overset{f}{\rightarrow} \mathbbm{k}B \hookrightarrow \mathbbm{k}H.\]
  All maps are morphisms of Hopf algebras, so $f$ restricts to a group homomorphism $\rho\colon A\to B$.  By construction, $f$, and therefore $\rho$, are bijections.  The choice of isomorphism $\mathbbm{k}^A\cong \mathbbm{k}A$ affects the composition, but for a given $p$ changing this isomorphism is equivalent to a change in the isomorphism $f$ (equiv. $\rho$).

  That \eqref{eq:vprel} implies $A\trianglelefteq G$ and $B^{v(G)}=B$ is clear.
	
	Now let $g\in C_G(A)$, and $a\in\mathbbm{k}^A$, $b\in B$ such that $p(a)=b$.  Then $b=p(a)=p(g\rightharpoonup a) = p(a)^{v(g)}=b^{v(g)}$.  Since $p$ surjects onto $B$, we conclude that $v$ maps $C_G(A)$ to $C_H(B)$.  The remaining claims are clear.
\end{proof}

In the subsequent, whenever discussing $p$ or any morphism $\morphquad \in \Hom(\D(G),\D(H))$, any use of the letters $A,B$ refer to precisely those subgroups in the preceding Lemma.  These subgroups will be used frequently, so whenever convenient we shall invoke and use them without further mention.  Note that since $A,B$ are abelian they are contained in their centralizers.

\begin{cor}\label{thetahopf1}
  Let $\morphquad\in\Hom(\D(G),\D(H))$. If $B\leq Z(H)$, then $u$ is a morphism of Hopf algebras.
\end{cor}
\begin{proof}
  Apply equations \eqref{eq:usplit} and \eqref{eq:vurel}.
\end{proof}
As remarked before, note that whenever $u$ is Hopf then we may identify $u^*$ with a group homomorphism.  We do so in the subsequent without further mention.

\begin{lem}\label{lem:factorizations}\label{lem:ABcentral}
	Let $\psi=\morphquad\in\Hom(\D(G),\D(H))$.  Then the following all hold.
	\begin{enumerate}
		\item $\Img(u^*)\subseteq C_{\mathbbm{k}G}(\mathbbm{k}A)$.
		\item If $\mathbbm{k}A \Img(u^*)=\mathbbm{k}G$, then $A\leq Z(G)$.
		\item If the images of $p,v$ commute, then $A\leq Z(G)$.
		\item If $A\leq Z(G)$ and $H=B\Img(v)$, then $B\leq Z(H)$.
		\item If $A\leq Z(G)$ and $H=Z(H)\Img(v)$, then $B\leq Z(H)$.
		\item If $\psi$ is surjective, then $B\Img(v)=H$.
		\item If $\psi$ is injective, then $\mathbbm{k}A \Img(u^*)=\mathbbm{k}G$ and $A\leq Z(G)$.
	\end{enumerate}
\end{lem}
\begin{proof}
	For the first statement, note that $u,p$ cocommute is equivalent to $u^*,p^*$ having commuting images.  Since $\Img(p^*)=\mathbbm{k}A$, the claim follows.  The second part then follows.  
	
	The third part is similar.  Explicitly, if $p,v$ commute, then equation \eqref{eq:vprel} becomes $p(e_g)=p(e_{xgx^{-1}})$ for all $x,g\in G$.  Since $p(e_g)\neq 0$ $\iff$ $g\in A$ and $p$ is an isomorphism $\mathbbm{k}^A\to\mathbbm{k}B$, we conclude that $A\leq Z(G)$, as desired.
	
	For the fourth part, we have
	\[H=Bv(G) = B v(C_G(A)) \subseteq B C_H(B) = C_H(B),\]
	and thus $B\leq Z(H)$ as claimed.  The fifth part is similar, using that $Z(H)\leq C_H(B)$.
	
	The surjectivity statement is an obvious consequence of equation \eqref{eq:morphpsi}.  Indeed, also from that equation, in order for $\psi(e_g\# g')\neq 0$ for all $g,g'\in G$, we must have that $\mathbbm{k}A\Img(u^*)=\mathbbm{k}G$.  Applying the previous parts completes the proof.
\end{proof}

The following is then an easy corollary.  Since it is essential for the rest of the paper, we mark it as a Theorem.
\begin{thm}\label{autcents}
  Let $\morphquad\in\Hom(\D(G),\D(H))$. If $\morphquad$ is an isomorphism then both $A,B$ are central, and $u$ is a morphism of Hopf algebras.
\end{thm}

Indeed, we have the following characterization of isomorphisms between $\D(G)$ and $\D(H)$.

\begin{thm}\label{thm:invariance}
	For finite groups $G,H$, we have $\D(G)\cong\D(H)$ $\iff$ $G\cong H$.
\end{thm}
\begin{proof}
	Sufficiency is clear.  For necessity, first note that any morphism of Hopf algebras sends group-like elements to group-likes.  Thus any isomorphism $\D(G)\cong\D(H)$ restricts to a group isomorphism $\widehat{G}\times G\cong \widehat{H}\times H$.  Since the groups are isomorphic, their character groups are isomorphic: $\widehat{G}^2 \cong \widehat{H}^2$.  By \citep[Exercise 6.32]{Rot99}, we conclude that $\widehat{G}\cong\widehat{H}$, and so $\widehat{G}\times G\cong \widehat{G}\times H$.  We may then apply \citep[Exercise 6.33]{Rot99} to conclude that $G\cong H$, as desired.
\end{proof}
In particular, we may focus on $\Aut(\D(G))$ without loss of generality.


\begin{cor}\label{cor:abelequiv}
Let $G$ be a finite group.  Then the following are equivalent.
\begin{enumerate}
  \item $G$ is abelian.
  \item There exists $\morphquad\in\Aut(\D(G))$ with $v\equiv 1$ (the trivial morphism).
  \item There exists $\morphquad\in\Aut(\D(G))$ with $u(e_g)=\delta_{1,g}\varepsilon$ (the trivial morphism).
\end{enumerate}
\end{cor}
\begin{proof}
That (ii) and (iii) imply (i) are a conequence of the Lemma.

On the other hand, if $G$ is abelian, then taking $v,u$ trivial and $p, r$ isomorphisms, we find that $\morphquad\in\Aut(\D(G))$.
\end{proof}

\begin{example}\label{ex:badabel}
	Indeed, not only does $G$ abelian allow for $u,v$ to have kernels, it also causes bijectivity to be sensitive to the choice of bicharacter.  Namely, for $G$ any abelian group, let $\morphquad\in\Aut(\D(G))$ be as in the proof of the Corollary.  If we replaced $r$ with any map which was not an isomorphism, then the new endomorphism is clearly not an automorphism.  It is not necessary that $p,r$ be isomorphisms for this to occur, in general.  
\end{example}
In both the Corollary and example, the reasons for the existence of such automorphisms are precisely the same as why the description of $\Aut(G\times G)$ depends on whether $G$ has abelian factors or not.  See \citep{BCM,Bi08} for details.

The following will be useful for describing some subgroups of $\Aut(\D(G))$ in Section \ref{sec:subgroups}.
\begin{lem}\label{lem:theta-aut1}
  Let $G$ be a finite group, and let $\psi=\morphquad\in\Aut(\D(G))$.  If either $p$ or $r$ is trivial, then both $u,v$ are isomorphisms.
	
	In particular, if $Z(G)=1$, then $p$ is always trivial and thus $u,v$ are always isomorphisms.
\end{lem}
\begin{proof}
  Consider the case $r$ trivial: $\ds{\morph{u}{0}{p}{v}}\in\Aut(\D(G)).$  Then for every $n\in\BN$ we have
	\[\morph{u}{0}{p}{v}^n = \morph{u^n}{0}{*}{v^n},\]
	where $*$ denotes some morphism.  By assumptions, the morphism in question has finite order, so for some $n\in\BN$ we have $u^n=\id$ and $v^n=\id$.  Therefore, $u,v$ are isomorphisms.  The case where $p$ is trivial is similar.
\end{proof}

We desire an alternative description of equation \eqref{eq:vurel}.  To this end, recall the following definition.

\begin{df}\label{def:normal}
	Let $H$ be a Hopf algebra with antipode $S$.  We say that an algebra morphism $f\colon H\to H$ is normal if $Sf*\id$ is an algebra morphism.  
\end{df}
Note that, since $f*(Sf*\id)=\id$, the definitions are symmetric: $Sf*\id$ and $f$ are simultaneously normal.  It is a routine verification that normality is equivalent to saying $f(a_{(1)}b S(a_{(2)})) = a_{(1)}f(b) S(a_{(2)})$.  The definition of normality here agrees with the definition used for group homomorphisms \citep{Rot99}.  

We can then provide the following equivalent characterization of equation \eqref{eq:vurel}.

\begin{lem}\label{lem:alphaisv}
	Let $G,H$ be finite groups, $u\colon \mathbbm{k}^{G\text{ cop}}\to \mathbbm{k}^{H\text{ cop}}$ a coalgebra morphism, and $v\colon \mathbbm{k}G\to \mathbbm{k}H$ an algebra morphism.  Then $u,v$ satisfy equation \eqref{eq:vurel} if and only if $u^*\circ v$ is a normal morphism of algebras.
\end{lem}
\begin{proof}
	First note that $u(x\rightharpoonup e_g) = v(x) \rightharpoonup u(e_g) \ \forall x,g\in G$ $\iff$ $u^*(v(x)\rightharpoonup h) = x\rightharpoonup u^*(h)$ $\forall h\in H,x\in G$.  Since $u^*$, and thus $u^*v$ are algebra maps, we then have
	\begin{align*}
		u^*(v(x)\rightharpoonup h) = x\rightharpoonup u^*(h) &\iff u^*v(x) u^*(h) u^*v(x^{-1}) = x u^*(h) x^{-1}\\
		&\iff u^*v(x^{-1})x u^*(h) (u^*v(x^{-1})x)^{-1} = u^*(h).
	\end{align*}
Then the latter holding for all $h$ is equivalent to $S(u^*v)*\id(x)=u^*v(x^{-1})x\in C_{\mathbbm{k}G}(\Img(u^*))$.  This being true for all $x$ is in turn equivalent to $S(u^*v)*\id$ being a morphism of algebras.  This completes the proof.	
\end{proof}
Of course, whenever $u$ is a morphism of Hopf algebras, then $u^*$ is identified with a group homomorphism, and the statement of the Lemma is then that $u^*\circ v$ is a normal group homomorphism.  When $u,v$ are also isomorphisms, the condition can be further phrased in group theoretic terms.

\begin{df}\label{def:group-centaut}
$\Aut_c(G) = \{ \phi\in\Aut(G) \ | \ \phi(g)g^{-1}\in Z(G)\}$ is a normal subgroup of $\Aut(G)$, called the central automorphism group. 
\end{df}
 $\Aut_c(G)$ can be equivalently characterized as the centralizer of $\Inn(G)$ in $\Aut(G)$.  It can also be characterized as the normal automorphisms of $G$.  Therefore, when $u^*,v\in\Aut(G)$ the lemma can be stated as saying $u^*\circ v\in \Aut_c(G)$.

Lemma \ref{lem:theta-aut1} gave some conditions that forced $u,v$ to be isomorphisms for $\morphquad\in\Aut(\D(G))$.  The theory of normal group endomorphisms allows us to show that $u,v$ are always isomorphisms for any purely non-abelian group.  We will see later that this is, in fact, a characterization of purely non-abelian groups.

\begin{thm}\label{thm:pure-uv-isoms}
	Let $G$ be a finite group and $\morphquad\in\Aut(\D(G))$.  Then $\ker(u^*)$ and $\ker(v)$ are contained in an abelian direct factor of $G$.  In particular, if $G$ is purely non-abelian then $u,v$ are isomorphisms for all $\morphquad\in\Aut(\D(G))$.
\end{thm}
\begin{proof}
	Let $\morphquad\in\Aut(\D(G))$, and denote its inverse by $(p',u',r',v')$.  Then
	\begin{align*}
		\morph{u'}{r'}{p'}{v'}\morph{u}{r}{p}{v} = \morph{u'u+r'p}{u'r+r'v}{p'u+v'p}{p'r+v'v} = \morph{1}{0}{0}{1}.
	\end{align*}
	In particular, $p'r+v'v=\id$. Thus for all $g\in\ker(v)$ we have $p'r(v)=v$.  Since $\Img(p')$ is a central group algebra, $p'v$ is clearly a normal endomorphism of $G$ with abelian image.  So by Fitting's Lemma we conclude that $G=\Img((p'v)^n)\times \ker((p'v)^n)$ for all sufficiently large $n\in\BN$.  Since $\ker(v)\subseteq \Img((p'v)^n)$ and $\Img((p'v)^n)$ is abelian, we conclude that $\ker(v)$ is contained in an abelian direct factor of $G$, as desired.  A similar argument applies to $u^*$ after dualizing the upper-left entries.  The remaining claim is by definition.
\end{proof}
\begin{rem}
	An alternative proof can be given using the methods of Section \ref{sec:restriction}.  This method also relies on the theory of normal group endomorphisms, however, so the above proof is more direct.
\end{rem}
This result will allow us to completely describe the elements of $\Aut(\D(G))$ for purely non-abelian groups in Theorem \ref{thm:main}.

\section{Flippable Homomorphisms}\label{sec:flip}

\begin{df}
	Let $G,H$ be finite groups and $\psi=\morphquad\in\Hom(\D(G),\D(H))$.  We say that $\psi$ is flippable if $\phi=(p^*,v^*,r^*,u^*)\in\Hom(\D(H),\D(G))$.  In this case we call $(p^*,v^*,r^*,u^*)$ the flip of $\psi$.
\end{df}
\begin{rem}
The taking of duals still gives morphisms with the proper domain and range, after using that $G$ is naturally isomorphic to $G^\text{op}$.  Indeed, the description of $\Hom(\D(G),\D(H))$ can have all co-opposites removed entirely.  Also observe that any flippable morphism necessarily has $u$ a morphism of Hopf algebras, since $u^*$ must necessarily be Hopf for $\phi$ to be a morphism of Hopf algebras.
\end{rem}
 We use the obvious nomenclature derived from the definition.  For example, the act of taking the flip will be called flipping, etc.  Clearly the flip of a flippable morphism is itself flippable.  The primary use for flippable morphisms in this paper is that proving general properties about $v$ for a collection of flippable morphisms, which is also closed under flipping, automatically lets one deduce the same properties for $u^*$.  The following lemma uses flipping to prove facts about $A,B$.

\begin{lem}\label{lem:flipcons}
	Let $G,H$ be finite groups, and suppose $\psi=\morphquad\in\Hom(\D(G),\D(H))$ is flippable.  Then the images of $p$ and $v$ commute.  Moreover, $A,B\leq Z(G)$.
\end{lem}
\begin{proof}
	Since $(p^*,v^*,r^*,u^*)\in\Hom(\D(H),\D(G))$, we know that $p^*$ and $v^*$ cocommute by \eqref{eq:upcocomm}.  The dual statement to this is that the images of $p$ and $v$ commute.  That $A\leq Z(G)$ then follows from Lemma \ref{lem:factorizations}.  By flipping, we obtain that $B\leq Z(H)$.
\end{proof}
In particular, when viewed as matrices the flippable morphisms have cocommuting columns and commuting rows.

\begin{prop}\label{prop:flipend}
	Let $G$ be a finite group and let $\End_f(\D(G))$ denote the set of all flippable endomorphisms of $\D(G)$.  Then $\End_f(\D(G))$ is a submonoid of $\End(\D(G))$, and flipping is an anti-isomorphism of monoids of order at most two.
\end{prop}
\begin{proof}
	Let $\phi,\psi\in\End_f(\D(G))$, and write
	\begin{align*}
		\psi &= \morph{u_1}{r_1}{p_1}{v_1},\\
		\phi &= \morph{u_2}{r_2}{p_2}{v_2}.
	\end{align*}
	Denote the flips of $\phi,\psi$ by $F(\phi),F(\psi)$ respectively.  By definition $F(F(\phi))=\phi$, which gives the order statement and bijectivity.
	
	Now observe that
	\begin{align}
		\phi\circ \psi &= \morph{u_2 u_1+ r_2 p_1}{u_2 r_1+ r_2 v_1}{p_2 u_1+v_2 p_1}{p_2 r_1+v_2 v_1},\label{eq:compose}\\
		F(\psi)\circ F(\psi) &= \morph{v_1^* v_2^*+r_1^* p_2^*}{v_1^* r_2^* + r_1^* u_2^*}{p_1^* v_2^* + u_1^* p_2^*}{p_1^* r_2^*+u_1^* u_2^*}\label{eq:flipcompose}.
	\end{align}
	Since the columns of $\psi$ cocommute, and the rows of $\phi$ commute, we conclude that $(u_2 u_1 + r_2 p_1)^* = (p_1^* r_2^* + u_1^* u_2^*)$, and similarly for the other entries of \eqref{eq:compose}.  Subsequently, we conclude that $F(\phi\circ\psi)=F(\psi)\circ F(\phi)$, and that $\phi\circ\psi\in\End_f(\D(G))$.  Since the identity map is clearly flippable, and is its own flip, this shows that $\End_f(\D(G))$ is a submonoid, and that $F$ is an anti-morphism of monoids.  This completes the proof.
\end{proof}
We now wish to show that every automorphism of $\D(G)$ is flippable, and that their flips are again automorphisms.  We need the following lemma.  The proof given is thanks to A. Caranti.  
\begin{lem}\label{lem:flipping-lem}
	Let $G,H$ be groups (not necessarily finite).  Let $v\colon G\to H$ and $w\colon H\to G$ be group homomorphisms.  Suppose $Z(H)\Img(v) = H$ and $Z(G)\Img(w) = G$.  Then the following hold:
	\begin{enumerate}
		\item $C_G(\Img(w)) = Z(G)$
	  \item $C_H(\Img(v)) = Z(H)$
		\item $v(Z(G))\subseteq Z(H)$
		\item $w(Z(H))\subseteq Z(G)$
		\item The following are equivalent:
		\begin{enumerate}
			\item $w\circ v$ is a normal group homomorphism.
			\item $v\circ w$ is a normal group homomorphism.
		\end{enumerate}
	\end{enumerate}
Indeed, if any of the two equivalent conditions hold, then also $\ker(v)\subseteq Z(G)$ and $\ker(w)\subseteq Z(H)$.
\end{lem}
\begin{proof}
		First, note that since $G = Z(G)\Img(w)$ we have
	\[Z(G) = C_G(G) = C_G(Z(G)\Img(w)) = C_G(\Img(w)).\]
	Similarly, $Z(H)=C_H(\Img(v))$.
	
	Now for any $g\in Z(G)$, we have $v(g)\in Z(\Img(v))$.  Since $H= Z(H)\Img(v)$, we have that $Z(\Img(v))\leq Z(H)$.  Thus $v(Z(G))\leq Z(H)$.  Similarly, $w(Z(H))\leq Z(G)$.
	
	We now need to prove the equivalence part.  By symmetry, we need only prove that (a)$\Rightarrow$(b). So suppose that (a) is true.  Let $h\in H$ be arbitrary, and write $h=bv(g)$ for some $b\in Z(G)$ and $g\in G$.  For simplicity of notation, write $v(g)=g^v$, and similarly for other group homomorphisms.  Then
	\[h^{wv}h^{-1} = (bg^v)^{wv} (bg^v)^{-1} = b^{wv}b^{-1} (g^{vw}g^{-1})^v \in Z(H),\]
	which is (b) by definition.  Here we used that $b^{wv}\in Z(H)$ by (iii) and (iv); and that $g^{vw}g^{-1}\in Z(G)$ by assumptions, and then again applying (iii).
	
	The remaining claims about $\ker(v)$ and $\ker(w)$ are then obvious.
\end{proof}

\begin{thm}\label{flipclosure}
	Every element of $\Aut(\D(G))$ is flippable.  Moreover, the flip of an automorphism is itself an automorphism, and flipping is an anti-isomorphism of $\Aut(\D(G))$.
\end{thm}
\begin{proof}
	By Theorem \ref{autcents} and Lemma \ref{lem:factorizations}, the preceding lemma applies to any $\morphquad\in\Aut(\D(G))$, with $w=u^*$.  The lemma is precisely what is needed for $\flipquad\in\End(\D(G))$.  Let $F$ denote the flip map.  Since $F$ is a monoidal anti-isomorphism of $\End_f(\D(G))$, it follows that $F(\morphquad^{-1})=F(\morphquad)^{-1}$, and in particular that $F$ restricts to an anti-isomorphism $\Aut(\D(G))\to\Aut(\D(G))$.
\end{proof}
\begin{rem}\label{rem:flipclosure}
	Lemma \ref{lem:flipping-lem} can also be used to show other homomorphisms are flippable, such as all morphisms when $G,H$ are abelian.  Automorphisms are just the simplest ones for which it can be verified the Lemma applies to in general.
\end{rem}
In general, flipping is not the same as inverting.  However, any group anti-isomorphism can be expressed as the composition of the inversion map and a group isomorphism.  This group isomorphism is simply the map $\psi\mapsto F(\psi^{-1})$.

Since all automorphisms are flippable, and flippable morphisms satisfy some of the basic requirements to be a bijection, we suspect there is a nice invertibility test that can be phrased in terms of flipping.  In general, flippable morphisms need not be invertible, as can easily be demonstrated by considering the endomorphisms of $\D(\BZ_2)$.  We suspect the following to be true.
\begin{conj}\label{invconj}
	$\psi\in\Hom(\D(G),\D(H))$ is invertible if and only if it is flippable and it sends the integral of $\D(G)$ to the integral of $\D(H)$.
\end{conj}
It is easy to verify that in order to preserve integrals, we must have $B\Img(v)=H$, and from being flippable we know $A,B$ are central.

If the conjecture were false, then there would exist finite groups $G,H$ and a non-bijective, flippable homomorphism $\psi\colon\D(G)\to\D(H)$ such that the induced functor $\Rep(\D(H))\to\Rep(\D(G))$ preserves all higher Frobenius-Schur indicators \cite{KSZ2}.  Of special interest would be the case $G=H$.  Morever, the induced functor for flippable homomorphisms tends to take a rather nice form, as will be described in Section \ref{sec:modules}.

\section{Some subgroups of \texorpdfstring{$\Aut(\D(G))$}{Aut(D(G))}}\label{sec:subgroups}
The goal of this section is to introduce several subgroups of $\Aut(\D(G))$.  We will use these in the subsequent sections to determine the membership and structure of $\Aut(\D(G))$ for purely non-abelian groups.

\begin{df}\label{def:lambdagroup}
	Let $G$ be a finite group, and set $\Lambda(G) = \{ (p,\id,1,\id)\in\Aut(\D(G)) \}$.
\end{df}
The choice of notation comes from the original description of morphisms $\D(G)\to\D(H)$ given by \citet{ABM}.  There, $p$ was written as $p(e_g)=\sum_{h\in H}\lambda(g,h)h$.

\begin{prop}\label{prop:lambdagroup}
	$\Lambda(G)$ is an abelian subgroup of $\Aut(\D(G))$, and $\Lambda(G)\cong \End(Z(G))$ as groups.
\end{prop}
\begin{proof}
We claim that $(p,\id,1,\id)\in\Aut(\D(G))$ $\iff$ $A,B\leq Z(G)$.

  Suppose that $(p,\id,1,\id)\in\Aut(\D(G))$.  Since the morphism is flippable, $p$ commutes and cocommutes with identity maps.  Thus $A$ and $B$ are central.
	
	On the other hand, suppose that $p$ is determined by $A,B\leq Z(G)$.  It is easy to verify that $(p,\id,1,\id)\in\End(\D(G))$.  In matrix notation, we observe that
		\[\morph{1}{0}{p}{1} \morph{1}{0}{-p}{1} = \morph{1}{0}{0}{1}.\]
		This means that $(p,\id,1,\id)$ is invertible with inverse $(Sp,\id,1,\id)$.  That $Sp$ is a morphism of Hopf algebras, and subsequently that \[(Sp,\id,1,\id)\in\End(\D(G))\], is guaranteed by the assumption that $A,B\leq Z(G)$.
	
This proves the claim.  It remains to show that $\Lambda(G)$ is an abelian group and to exhibit the desired isomorphism.
	
Let $(p_1,\id,1,\id),(p_2,\id,1,\id)\in\Lambda(G)$.  Then the composition is
	\[ \morph{1}{0}{p_1}{1} \morph{1}{0}{p_2}{1} = \morph{1}{0}{p_1+p_2}{1}.\]
	Equivalently, $(p_1,\id,1,\id)\circ(p_2,\id,1,\id) = (p_1*p_2,\id,1,\id)$.  It follows that $\Lambda(G)$ is a subgroup of $\Aut(\D(G))$.	As previously noted, $p_1$ and $p_2$ commute and cocommute with identity maps.  Therefore they convolution commute: $p_1*p_2=p_2*p_1$.  Therefore $\Lambda(G)$ is abelian.
	
	The claim in fact shows that $\Lambda(G)$ depends only on $Z(G)$, and in particular $\Lambda(G)\cong \Lambda(Z(G))$.  So we may suppose that $G$ is abelian.  Then a morphism of Hopf algebras $p\colon \mathbbm{k}^G\to \mathbbm{k}G$ is exactly a group homomorphism $\widehat{G}\to G$.  Fixing an isomorphism $\widehat{G}\cong G$, any group homomorphism $\widehat{G}\to G$ is equivalent to a group homomorphism $G\to \widehat{G}$, as well as equivalent to a group homomorphism $G\to G$.  It is easy to see that these in fact give group isomorphisms.
	
	We conclude that \[\Lambda(G)\cong \Hom(\widehat{Z(G)},Z(G))\cong\BCh{Z(G)}\cong \End(Z(G))\] for all finite groups $G$.  This completes the proof.
\end{proof}

\begin{prop}\cite[Proposition 2.3.5]{Co}
	Let $G$ be a finite group.  Then $\Aut(G)$ embeds as a subgroup of $\Aut(\D(G))$, where $\sigma\in\Aut(G)$ is sent to the automorphism defined by $e_g\# h\mapsto e_{\sigma(g)}\#\sigma(h)$.  Equivalently, \[\sigma\mapsto \morph{(\sigma^{-1})^*}{0}{0}{\sigma}.\]
\end{prop}

In particular, we identify $\Aut(G)$ as a subgroup of $\Aut(\D(G)$ via the image of this embedding. 

\begin{prop}\label{prop:bch-auts}
	For any $r\in\BCh{G}$,
	\[\morph{1}{r}{0}{1}\in \Aut(\D(G)).\]
	In particular, this gives an embedding $\BCh{G}\hookrightarrow\Aut(\D(G))$, sending $r\in\BCh{G}$ to the map $e_g\# h \mapsto r(h)e_g\# h$.
\end{prop}
\begin{proof}
	Given $r\in\BCh{G}$, we see the desired matrix is trivially an element of $\End(\D(G))$.  If $r$ has order $n$, we can easily see that the matrix also has order $n$, so that the matrix is an element of $\Aut(\D(G))$ as claimed.  It is another simple check that the claimed embedding is actually an injective group homomorphism.
\end{proof}
We identify $\BCh{G}$ with the image of this embedding.

\begin{rem}\label{rem:r-to-omega}
	If we write $r(h)=\sum \omega(h,x)e_x$, where $\omega$ is some bicharacter, then the description of the map is \[e_g\# h \mapsto \omega(h,g)e_g \# h.\]
\end{rem}

As $\Aut_c(G)$ is a subgroup of $\Aut(G)$, it can be viewed as a subgroup of $\Aut(\D(G))$ through the embedded copy of $\Aut(G)$.  However, in light of Lemma \ref{lem:alphaisv}, there are additional copies of $\Aut_c(G)$.
\begin{prop}\label{prop:autc-embed}
	For any $w\in\Aut_c(G)$ we have \[\morph{1}{0}{0}{w}, \morph{(w^{-1})^*}{0}{0}{1}\in\Aut(\D(G)).\]
	
	In particular, these give two distinct, injective group homomorphisms $\Aut_c(G)\to\Aut(\D(G))$ whose images intersect trivially.
\end{prop}
Indeed, the embeddings also intersect trivially with the usual embedding of $\Aut(G)$, and corresponding sub-embedding of $\Aut_c(G)$.  The second embedding will be the one we use in the subsequent.  So when we refer to $\Aut_c(G)$ as a subgroup of $\Aut(\D(G))$, it will be to this second embedding.  To help avoid potential confusion with these three distinct copies of $\Aut_c(G)$, we will prefer, whenever possible, to consider the subgroup generated by $\Aut(G)$ and either of the embeddings of $\Aut_c(G)$ above instead.

\begin{df}\label{def:spautc}
	Let $G$ be any group (not necessarily finite).  We define $\SpAutc(G)$, the split central automorphism group of $G$, by
	\[\SpAutc(G) = \{ (w,v)\in\Aut(G)\times\Aut(G) \ \colon \ w^{-1}\circ v\in\Aut_c(G) \}.\]
	The binary operation is inherited from $\Aut(G)\times \Aut(G)$.
\end{df}
It may not be immediately clear that this is actually a subgroup, so we justify the terminology with the following result.

\begin{lem}\label{lem:spautc}
  $\SpAutc(G)$ is a subgroup of $\Aut(G)\times\Aut(G)$.  Furthermore, \[\SpAutc(G)\cong\Aut_c(G)\rtimes\Aut(G)\] with the usual conjugation action.
\end{lem}
\begin{proof}
  We note that $(\id,\id)\in\SpAutc(G)$, and this element serves as a two-sided identity.  Now suppose $(\alpha,v)\in\SpAutc(G)$.  The inverse of this element would be $(\alpha^{-1},v^{-1})$, so we must show this is an element of $\SpAutc(G)$.  Since $\Aut_c(G)$ is a normal subgroup of $\Aut(G)$, we have that
  \begin{align*}
    (\alpha,v)\in\SpAutc(G) &\iff \alpha^{-1}v\in\Aut_c(G)\\
    &\iff v\in\alpha\Aut_c(G)\\
    &\iff v\in\Aut_c(G)\alpha\\
    &\iff v\in(\alpha^{-1}\Aut_c(G))^{-1}\\
    &\iff v^{-1}\in\alpha^{-1}\Aut_c(G)\\
    &\iff (\alpha^{-1},v^{-1})\in\SpAutc(G).
  \end{align*}
  which establishes that $\SpAutc(G)$ is closed under inversion.  If also $(\beta,w)\in\SpAutc(G)$, then we may write $v(g)=\alpha(g)z(g)$ and $\beta^{-1}(w(g))=g z'(g)$ for some $z,z'\in\Hom(G,Z(G))$; we note that $z,z'$ are homomorphisms, which follows from $Z(G)$ being abelian.  We then have
  \begin{align*}
    (\alpha\circ\beta)^{-1}\circ(v\circ w)(g) &= \beta^{-1}\circ\alpha^{-1}\circ v\circ w(g)\\
    &= \beta^{-1}\alpha^{-1}\circ \alpha*z\circ w(g)\\
    &= \beta^{-1}(\alpha^{-1}(\alpha(w(g))z(w(g))))\\
    &= \beta^{-1}(w(g)\alpha^{-1}(z(w(g))))\\
    &= \beta^{-1}(w(g))\alpha^{-1}(z(w(g)))\\
    &= g z'(g)\beta^{-1}(\alpha^{-1}(z(w(g)))).
  \end{align*}
  Since $Z(G)\Char G$, we conclude that $\beta^{-1}(\alpha^{-1}(z(w(g))))\in Z(G)$ for all $g$, and therefore that $(\alpha\circ\beta,v\circ w)\in\SpAutc(G)$, as desired.
	
	Consider $H=\{(w,1)\in\SpAutc(G)\}$ and $K=\{(v,v)\in\SpAutc(G)\}$.  By definition we have $H\cong\Aut_c(G)$ and $K\cong\Aut(G)$.  Furthermore, it is equally clear that $H\trianglelefteq\SpAutc(G)$, $H\cap K=1$, and $\SpAutc(G)=HK$.  Thus $\SpAutc(G)\cong\Aut_c(G)\rtimes\Aut(G)$, as desired.

  This complete the proof.
\end{proof}
We view $\Aut(G)$ as a subgroup of $\SpAutc(G)$ via the embedding $v\mapsto (v,v)$.  Since there will be many semidirect products involved in our investigation of $\Aut(\D(G))$, and $\Aut(\D(G))$ has many natural copies of $\Aut_c(G)$, for clarity of presentation we prefer to use $\SpAutc(G)$ rather than $\Aut_c(G)\rtimes \Aut(G)$ to denote this group.  We use the latter only for specific examples.  Whenever $\Aut_c(G)=1$, such as when $Z(G)=1$, we clearly have $\SpAutc(G)\cong\Aut(G)$.

\begin{prop}\label{prop:spautc-auts}
	Let $G$ be a finite group.  For any $(w,v)\in\SpAutc(G)$,
	\[\morph{(w^{-1})^*}{0}{0}{v}\in\Aut(\D(G)).\]
	In particular, this gives an embedding $\SpAutc(G)\hookrightarrow\Aut(\D(G))$.  Explicitly, $(w,v)$ corresponds to the isomorphism $e_g\# h \mapsto e_{w(g)}\# v(h) = (w^{-1})^*(e_g)\# v(h).$
\end{prop}
\begin{proof}
	Such a matrix is easily verified to give an element of $\End(\D(G))$.  Moreover, it also clear that its inverse is the matrix corresponding to $(w^{-1},v^{-1})=(w,v)^{-1}$, and so the matrix is in $\Aut(\D(G))$.  It is equally clear that this gives an injective group homomorphism.
\end{proof}
It should be noted that the embedded copy of $\Aut(G)$ in $\SpAutc(G)$ is mapped to the usual embedding of $\Aut(G)$ into $\Aut(\D(G))$.  So the copy of $\SpAutc(G)$ is naturally viewed as a generalization of this embedding.  Indeed, it is clear that this copy of $\SpAutc(G)$ contains both of the embeddings from Proposition \ref{prop:autc-embed}, and moreover that it is generated by $\Aut(G)$ and either one of these copies of $\Aut_c(G)$.  As usual, we identify $\SpAutc(G)$ with the image of this embedding.

If we combine Propositions \ref{prop:bch-auts} and \ref{prop:spautc-auts}, we get a subgroup containing, and generated by, both $\BCh{G}$ and $\SpAutc(G)$.
\begin{thm}\label{thm:bcg-auts}
  Let $G$ be a finite group.  
	
	Then the elements of $\Aut(\D(G))$ of the form
	\[\morph{u}{r}{0}{v}\]
	form a subgroup isomorphic to $\BCGG$, where the action of $\SpAutc(G)$ on $\BCh{G}$ is given by $r^{(w,v)} = w^* r v$. 
	
	Explicitly, $((w,v),r)\in\BCGG$ corresponds to the isomorphism $e_g\# h \mapsto r(h)e_{w^{-1}(g)}\# v(h)$.
\end{thm}
\begin{proof}
	Note that by Lemma \ref{lem:theta-aut1} we are guaranteed that $u,v$ are isomorphisms, and $u^*\circ v\in\Aut_c(G)$.  Given this, the rest of the proof is another straightforward verification using the matrix form for elements of $\Aut(\D(G))$, and the preceding results.  In particular, we can write
	\begin{align*}
		\left( \begin{array}{cc} u& r\\ 0&v \end{array}\right) &= \left( \begin{array}{cc} 1& rv^{-1}\\ 0& 1\end{array}\right)\left( \begin{array}{cc} u&0 \\ 0&v \end{array}\right)\\
		&= \left( \begin{array}{cc} u&0 \\ 0&v \end{array}\right)\left( \begin{array}{cc} 1& u^{-1}r\\ 0& 1\end{array}\right),
	\end{align*}
from which it follows that
\begin{align*}
	\left( \begin{array}{cc} u^{-1}& 0\\ 0& v^{-1}\end{array}\right) \left( \begin{array}{cc} 1& r\\ 0& 1\end{array}\right) \left( \begin{array}{cc} u& 0\\ 0&v \end{array}\right) &= \left( \begin{array}{cc} 1& u^{-1}rv\\ 0& 1\end{array}\right).
\end{align*}
\end{proof}
Note that $\BCGG$ is a proper subgroup whenever $Z(G)\neq 1$, since then it intersects trivially with the non-trivial subgroup $\Lambda(G)$. On the other hand, $\BCGG\cong\BCG$ and $\Lambda(G)=1$ whenever $Z(G)=1$.

\section{Main Result}\label{sec:main}

We can now proceed to establish the main result of the paper, which characterizes all elements of $\Aut(\D(G))$ for $G$ purely non-abelian, and provides a decomposition into a product of subgroups.  First, we need a few preliminary results.  The first concerns the obvious subgroup of $\Aut(\D(G))$ when $G$ is a direct product.

\begin{prop}\label{prop:prodsub}
	Let $G=H\times K$, with $H,K$ finite groups.  Then $\Aut(\D(H))\times\Aut(\D(K))$ is (isomorphic to) a subgroup of $\Aut(\D(G))$.
\end{prop}
\begin{proof}
	Given 
	\begin{align*}
		f&=(p_H,u_H,r_H,v_H)\in\Aut(\D(H)),\\
		g&=(p_K,u_K,r_K,v_K)\in\Aut(\D(K)),
	\end{align*}
	then clearly $f\otimes g=(p_H\otimes p_K, u_H\otimes u_K,r_H\otimes r_K, v_H\otimes v_K)\in\Aut(\D(G))$.  This is precisely the automorphism of $\D(G)$ which restricts to $f$ on $\D(H)$ and restricts to $g$ on $\D(K)$.
\end{proof}
By Corollary \ref{cor:abelequiv}, example \ref{ex:badabel}, and Theorem \ref{thm:pure-uv-isoms} we then have the following two corollaries.

\begin{cor}\label{cor:pure-all-isoms}
	$G$ is purely non-abelian if and only if $u,v$ are isomorphisms for all $\morphquad\in\Aut(\D(G))$.
\end{cor}
\begin{cor}\label{cor:abel-cons}
	If $G$ has an abelian direct factor, then the following all hold:
	\begin{enumerate}
		\item $\exists\morphquad\in\Aut(\D(G))$ with $\ker(v)\neq 1$.
		\item $\exists\morphquad\in\Aut(\D(G))$ with $\ker(u^*)\neq 1$.
		\item $\exists\morphquad\in\Aut(\D(G))$ and $r'\in\Hom(G,\widehat{G})$ such that \[(p,u,r',v)\not\in\Aut(\D(G)).\]
		\item $\exists\morphquad\in\Aut(\D(G))$ with any two, or all three, of the above properties.
	\end{enumerate}
\end{cor}

We previously defined the subgroup $\Lambda(G)$ in Definition \ref{def:lambdagroup}.  We will need the following lemma.

\begin{lem}\label{lem:pfactors}
	Let $\ds{\left( \begin{array}{cc} u& r\\ p& v\end{array} \right)}\in\Aut(\D(G))$. If $u$ is invertible, then we may write
	\begin{align*}
		\left( \begin{array}{cc} u& r\\ p& v\end{array} \right) &= \left( \begin{array}{cc} 1& 0\\ pu^{-1}& 1\end{array} \right) \left( \begin{array}{cc} u& r\\ 0& -pu^{-1}r+v\end{array} \right).
	\end{align*}
Here the left matrix is in $\Lambda(G)$, and the right is in $\BCGG$.

On the other hand, if $v$ is invertible, then 
\begin{align*}
		\left( \begin{array}{cc} u& r\\ p& v\end{array} \right) &= \left( \begin{array}{cc} u-rv^{-1}p& r\\ 0& v\end{array} \right) \left( \begin{array}{cc} 1&0 \\ v^{-1}p& 1\end{array} \right),
	\end{align*}
	and this is again a product of automorphisms, this time with an element of $\Lambda(G)$ on the right and an element of $\BCGG$ on the left.
\end{lem}
\begin{proof}
	Simply multiply out the matrices and use Proposition \ref{prop:lambdagroup} to obtain the products.  We need only justify that the second matrices are in $\BCGG$.  To this end, note that for either such matrix we have $0$ in the bottom left corner, and thus they have $A=B=1$.  Since the matrix gives an isomorphism, by Lemma \ref{lem:factorizations} we conclude that the diagonal terms are isomorphisms.  This completes the proof.
\end{proof}
\begin{rem}
	An alternative characterization of the lemma is that any automorphism $\morphquad$ with either $u$ or $v$ an isomorphism lies in the subgroup generated by $\Lambda(G)$ and $\BCGG$.
\end{rem}
\begin{example}
	In general, it need not be true that $u,v$ are always simultaneously isomorphisms.  Indeed, for $G=\BZ_2$, denote its non-trivial bicharacter by $\sgn$, and the unique isomorphism $\widehat{G}\to G$ by $p$.  Then the following is in $\Aut(\D(G))$:
	\begin{align*}
		\morph{0}{\sgn}{p}{1}.
	\end{align*}
Indeed, $-\sgn p$ is the isomorphism $g\mapsto g^{-1}$.
\end{example}

\begin{cor}\label{cor:sp-lambda}
	All elements of $\Aut(\D(G))$ of the form $\ds{\morph{u}{0}{p}{v}}$ form the subgroup $\Lambda(G)\rtimes \SpAutc(G)$.  Here $(w,v)\in\SpAutc(G)$ acts on $\Lambda(G)$ from the left by $(w,v).(p,\id,1,\id)=(vpw^*,\id,1,\id)$.
\end{cor}
\begin{proof}
	By Lemma \ref{lem:theta-aut1}, any such morphism has $u,v$ isomorphisms.  It is easily verified that such morphisms form a subgroup, and by the Lemma the subgroup is the product of the subgroups $\Lambda(G)$ and $\SpAutc(G)$.  In particular, we have
	\begin{align*}
		\morph{u}{0}{0}{v} \morph{1}{0}{p}{1} \morph{u^{-1}}{0}{0}{v^{-1}} &= \morph{1}{0}{vpu^{-1}}{1}.
	\end{align*}
	This completes the proof.
\end{proof}

Recall that an (exact) factorization of a group $G$ is given by subgroups $H,K$ such that $G=HK$ and $H\cap K=1$ \citep{Ta81,WW80,Ito51}.  Note that neither subgroup is required to be normal.  

Our main result also now follows directly from the Lemma and Theorem \ref{thm:pure-uv-isoms}.
\begin{thm}\label{thm:main}
	Let $G$ be a purely non-abelian finite group.  Then we have exact factorizations 
	\begin{align*}
		\Aut(\D(G)) &= \Lambda(G)(\BCGG)\\&=(\BCGG)\Lambda(G).
	\end{align*}
	We therefore also have the order formula
	\[|\Aut(\D(G))| = |\End(Z(G))|\cdot |\BCh{G} | \cdot |\Aut_c(G)| \cdot |\Aut(G)|.\]
\end{thm}

Recall that a group is said to be perfect if $G=G'$, the derived subgroup of $G$.
\begin{cor}\label{cor:perfect}
	Suppose $G$ is a perfect group.  Then \[\Aut(\D(G))\cong\Lambda(G)\rtimes\Aut(G).\]
	Here the action is $v.(p,\id,1,\id) = (v p v^*,\id,1,\id)$.
	
	In particular, $G$ is centerless and perfect $\iff$ $\Aut(\D(G))\cong\Aut(G)$.  This latter isomorphism therefore holds for all non-abelian simple groups.
\end{cor}
\begin{proof}
	A perfect group is clearly purely non-abelian.  Morever, since $G=G'$ we have $\BCh{G}=1$.  Therefore $\Aut(\D(G))\cong\Lambda(G)\rtimes\SpAutc(G)$ by Corollary \ref{cor:sp-lambda}.  It is well known that, for any group $G$, the members of $\Aut_c(G)$ fix $G'$ elementwise.  Thus for $G$ perfect $\Aut_c(G)=1$ and $\SpAutc(G)\cong\Aut(G)$.  Therefore $\Aut(\D(G))\cong\Lambda(G)\rtimes\Aut(G)$ as claimed.
\end{proof}

\begin{cor}\label{cor:centerless}
	If $Z(G)=1$, then $\Aut(\D(G))\cong\BCG$.
\end{cor}

\begin{example}
	Consider $S_n$, the symmetric group on $n\geq 3$ symbols.  Then \[\Aut(\D(S_n))\cong\BZ_2\times \Aut(S_n).\]
\end{example}
\begin{example}
	For $A_n$, the alternating group on $n\geq 5$ symbols, \[\Aut(\D(A_n))\cong\Aut(S_n).\]
\end{example}
\begin{example}
  For $A_4$ we have $\BCh{A_4}\cong\BZ_3$ and thus \[\Aut(\D(A_4))\cong S_4\ltimes\BZ_3,\] where transpositions act by inversion.
\end{example}
\begin{example}\label{ex:dihedrals1}
  Let $D_{2n}$ be the dihedral group of order $2n$.  If $n$ is odd, then $\Aut(\D(D_{2n}))\cong \BZ_2\times \operatorname{Hol}(\BZ_n)$, where $\Aut(D_{2n})\cong\operatorname{Hol}(\BZ_n)$ is the holomorph of $\BZ_n$.  In this case, $|\Aut(\D(D_{2n}))|=2n\phi(n)$, where $\phi$ is the Euler totient function.  When $n$ is even the center is non-trivial, and the description becomes more complicated.  Indeed, when $n$ is even and not divisible by 4, $D_{2n}\cong\BZ_2\times D_{n}$, and so $D_{2n}$ is not purely non-abelian.  We leave further details on this to the next section.
\end{example}
\begin{example}
	For $q\geq 4$ a prime power, $G=\operatorname{SL}(2,q)$ is quasisimple, so in particular perfect, and has $\Aut(G)\cong\operatorname{P\Gamma L}(2,q)$, a projective semilinear group.  Therefore \[\Aut(\D(\operatorname{SL}(2,q)))\cong \End(\BZ_q)\rtimes \operatorname{P\Gamma L}(2,q).\]  When $q$ itself is prime, note that $\End(\BZ_q)\cong \BZ_q$ and $\operatorname{P\Gamma L}(2,q)=\operatorname{PSL}(2,q)$.
\end{example}

We can summarize our results on purely non-abelian groups with the following.
\begin{thm}\label{thm:pure-equiv}
Let $G$ be a finite group.  Then the following are equivalent.
\begin{enumerate}
	\item $G$ is purely non-abelian.
	\item Every $\morphquad\in\Aut(\D(G))$ has $u,v$ isomorphisms.
	\item For any $r'\in\Hom(G,\widehat{G})$ and $\morphquad\in\Aut(\D(G))$, then also $(p,u,r',v)\in\Aut(\D(G))$.
	\item $\Aut(\D(G))=\Lambda(G)(\BCGG)$ is an exact factorization.
\end{enumerate}
\end{thm}

Our next goal is to find a number of normal subgroups of $\Aut(\D(G))$ and determine when they have complements.

\section{Restriction to group-likes}\label{sec:restriction}

Let $\Gamma=\Gamma_G=\widehat{G}\times G$ be the group-like elements of $\D(G)$.  Since any morphism of Hopf algebras must send group-like elements to group-like elements, there is a natural restriction of $\Hom(\D(G),\D(H))$ to the group homomorphisms $\Hom(\Gamma_G,\Gamma_H)$.  In particular, this gives a group homomorphism $\Aut(\D(G))\to \Aut(\Gamma)$. This simple observation was used earlier in the proof of Theorem \ref{thm:invariance}.

Here and for the remainder of the paper, let $N$ be the kernel of the restriction $\Aut(\D(G))\to\Aut(\Gamma)$, and identify $\Aut(\D(G))/N$ with the image. In particular, $N$ is the normal subgroup of all automorphisms of $\D(G)$ which fix every element of $\Gamma$.  When $N=1$ this means that $\Aut(\D(G))$ is isomorphic to a subgroup of $\Aut(\Gamma)$.  In this section we describe $N$ explicitly, and determine when $N=1$ holds.  In the next section we will determine when $N$ and certain other subgroups have a complement.

It is helpful to know exactly what the restriction map looks like.

\begin{lem}\label{lem:restriction}
	Let $\psi=\morphquad\in\Hom(\D(G),\D(H))$. Then for $\chi\in\widehat{G}$ and $g\in G$ we have
	\begin{align}\label{eq:res}
		\psi(\chi\# g) = r(g)u(\chi) \# p(\chi) v(g).
	\end{align}
	Thus the restriction of $\psi$ gives the group homomorphism
	\begin{align}\label{eq:res-group}
		(\chi,g) \mapsto (r(g)u(\chi), p(\chi)v(g)),
	\end{align}
	and this will be an isomorphism whenever $\psi$ is.
\end{lem}
Note that since $u$ is a unitary coalgebra map, it restricts to a group homomorphism $\widehat{G}\to\widehat{H}$.

In general, however, $\psi$ need not be an isomorphism for its restriction to be an isomorphism.  
\begin{example}
	If $G$ is a perfect group, then $\ds{\morph{0}{0}{0}{1}}\in\End(\D(G))$ is clearly not an isomorphism, but restricts to the identity on $\Gamma_G=G$.
\end{example}

We now proceed to determine the membership and structure of $N$.  We have previously defined $\Lambda(G)$ in Definition \ref{def:lambdagroup}.
\begin{df}\label{def:lambdac}
For $G$ a finite group, define
\begin{align*}
	\Lambda_c(G)=\{(p,u,1,\id)\in\Aut(\D(G)) \}.
\end{align*}
\end{df}
Clearly $\Lambda(G)\subseteq\Lambda_c(G)$.  For $(p,u,1,\id)\in\Lambda_c(G)$ we have $u^*=u^*\circ\id$ is a normal group endomorphism.  In particular, if $u^*\in\Aut(G)$, then in fact $u^*\in\Aut_c(G)$.  Furthermore, for $u^*\in\Aut_c(G)$,
\[(p,u,1,\id)=(1,u,1,\id)\circ(p,\id,1,\id)\]
is a composition of automorphisms by Propositions \ref{prop:lambdagroup} and \ref{prop:spautc-auts}, and so is itself an automorphism. It may not be immediately obvious that $u$ must be an isomorphism for any element of $\Lambda_c(G)$, however.

\begin{lem}
Every $(p,u,1,\id)\in\Lambda_c(G)$ has $u^*\in\Aut_c(G)$.  Moreover, $\Lambda_c(G)$ is the subgroup of $\Aut(\D(G))$ fixing $1\times G\subseteq \Gamma$ element-wise.
\end{lem}
\begin{proof}
	Let $(p_1,u_1,1,\id),(p_2,u_2,1,\id)\in\Lambda_c(G)$.  Then
	\begin{align*}
		\left( \begin{array}{cc} u_1& 0 \\ p_1 & 1 \end{array} \right) \left( \begin{array}{cc} u_2& 0 \\ p_2 & 1 \end{array} \right) &= \left( \begin{array}{cc} u_1 u_2 & 0 \\ p_1 u_2+p_2 & 1 \end{array} \right).
	\end{align*}
	Since $(p_1,u_1,1,\id)$ is invertible and has finite order, we see that we must have $u_1^n=\id$ for some $n\in\BN$.  Thus $u$ must be invertible.  This proves the first claim, and the matrix product above then shows that $\Lambda_c(G)$ is a subgroup.
	
	Now, on the other hand, let $g\in G$ and $\psi=\morphquad\in\Aut(\D(G))$.  Then by \eqref{eq:res},
	\[ \psi(\varepsilon\# g) = r(g)\# v(g).\]
	Thus $\psi$ fixes all of $G$ if and only if $r$ is trivial and $v$ is the identity.  So by definition, if and only if $\psi\in \Lambda_c(G)$.  This completes the proof.	
\end{proof}
\begin{cor}
	$\Lambda_c(G)\cong \Aut_c(G)\ltimes \Lambda(G)$ where $\Aut_c(G)$ acts on $\Lambda(G)$ on the right by $(p,\id,1,\id)^w =(p(w^{-1})^*,\id,1,\id)$. 
\end{cor}
\begin{proof}
 We embed $\Aut_c(G)$ into $\Lambda_c(G)$ (and $\Aut(\D(G))$) in particular) as in Proposition \ref{prop:autc-embed}: \[w\mapsto \morph{(w^{-1})^*}{0}{0}{1}.\]  Then we compute
	\begin{align*}
		\morph{w^*}{0}{0}{1}\morph{1}{0}{p}{1}\morph{(w^{-1})^*}{0}{0}{1} &= \morph{1}{0}{p(w^{-1})^*}{1}.
	\end{align*}
Thus $\Lambda(G)\trianglelefteq \Lambda_c(G)$.  Clearly $\Aut_c(G)$ and $\Lambda(G)$ intersect trivially, and by Lemma \ref{lem:pfactors} we conclude that $\Lambda_c(G)$ is the product of these subgroups.  This completes the proof.
\end{proof}

We now consider which elements of $\Lambda_c(G)$ fix $\widehat{G}$, which by construction gives us the normal subgroup $N$.

Consider $\psi\in\Lambda_c(G)$ and let $(\chi,g)\in\Gamma$.  Then $\psi(\chi\# g) = u(\chi)\# p(\chi)g$.  Thus $\psi$ fixes all of $\Gamma$ if and only if $p$ is trivial on $\widehat{G}$, and $u$ is the identity on $\widehat{G}$.  Since $u^*\in\Aut_c(G)$, we see that $\chi=u(\chi) = \chi\circ u^*$ for all $\chi\in \widehat{G}$ $\iff$ $u^*(g^{-1})g\in Z(G)\cap G'$ for all $g\in G$.
\begin{df}\label{def:autcp}
	Let $G$ be any group, and define
	\[\Aut_{c'}(G) = \{\phi\in\Aut_c(G) \ | \ \phi(g^{-1})g\in Z(G)\cap G' \text{ for all } g\in G\}.\]
\end{df}
It is easy to show that $\Aut_{c'}(G)$ is a normal subgroup of $\Aut_c(G)$ and also $\Aut(G)$.  What we have just shown is that $u^*\in\Aut_{c'}(G)$.

Furthermore, since $p$ is given by an isomorphism from $\widehat{A}\to B$, we conclude that $p$ is trivial on $\widehat{G}$ $\iff$ $\chi|_A=\varepsilon_A$ for all $\chi\in \widehat{G}$, which is in turn equivalent to $A\leq G'$.  Since necessarily $A\leq Z(G)$, this is equivalent to $A\leq Z(G)\cap G'$.  Any such $p$ can be uniquely identified with a member of the group $\Hom(Z(G)\cap G',Z(G))$ under the identification of $p$ with an element of $\End(Z(G))$, and vice versa.  Moreover, this respects the group structures.

Combining, we get the desired description of $N$.

\begin{thm}\label{thm:ndesc}
	$N\cong \Aut_{c'}(G)\ltimes \Hom(Z(G)\cap G',Z(G))$.  Thus, $N=1$ if and only if $Z(G)\cap G'=1$.  In particular, if $G$ is abelian or centerless, then $N=1$.  At the other extreme, $Z(G)\leq G'$ $\iff$ $N=\Lambda_c(G)$, and the kernel in this case is non-trivial precisely when $Z(G)\neq 1$.
\end{thm}

Recall that a group $G$ is called a stem group if $Z(G)\leq G'$ \citep{Ha40}.  Note that any stem group is necessarily purely non-abelian.  By Lemma \ref{lem:pfactors} we get the following.
\begin{cor}\label{cor:stemquotient}
	If $G$ is a stem group, then \[\Aut(\D(G))/N\cong\BCG.\]
\end{cor}
In the next section we will show that \[\Aut(\D(G))\cong N\rtimes(\BCG)\] precisely when $G$ is a stem group.

\begin{example}\label{ex:dihedrals2}
	$N=1$ for $D_{2n}$, the dihedral group of order $2n$, if and only if $n\not\equiv 0\mod 4$.  On the other hand, $D_{2n}$ is a stem group whenever $n$ is divisible by 4, and so the corollary applies.  See also Examples \ref{ex:dihedrals1} and \ref{ex:dihedrals3}.
\end{example}
\begin{example}\label{ex:fixderived}
	It is well known that, for any group, the elements of $\Aut_c(G)$ fix $G'$ elementwise.  So let $G$ be a perfect group.  Then $G$ is trivially a stem group, $\BCh{G}=1$, and we have $N=\Lambda(G)\cong\End(Z(G))$.  Therefore $\Aut(\D(G))/N\cong\Aut(G)$.  This also follows from Corollary \ref{cor:perfect}, which gives a stronger result.
\end{example}

\section{Exact factorizations and semidirect products for \texorpdfstring{$\Aut(\D(G))$}{Aut(D(G))}}\label{sec:products}

We now wish to expand our ability to find exact factorizations for $\Aut(\D(G))$, and to determine when they are semidirect products.  These results may be considered as variations on the main theorem.

We have already shown that $\Lambda(G)$ and $\SpAutc(G)$ give an exact factorization of $\Aut(\D(G))$ if and only if $G$ is purely non-abelian.  The next result determines when this is actually a semidirect product.
\begin{thm}\label{thm:stem-split1}
	$\Lambda(G)$ is normal in $\Aut(\D(G))$ $\iff$ $G$ is a stem group.  Therefore, $\Aut(\D(G))\cong\Lambda(G)\rtimes(\BCGG)$ $\iff$ $G$ is a stem group.
\end{thm}
\begin{proof}
	Let $\ds{\morph{u}{r}{0}{v}}\in\BCGG\subseteq\Aut(\D(G))$.  Then its inverse is $\ds{\morph{u^{-1}}{-u^{-1}r v^{-1}}{0}{v^{-1}}}$.  Therefore the conjugate
	\[ morph{u}{r}{0}{v} \morph{1}{0}{p}{1} \morph{u^{-1}}{-u^{-1}r v^{-1}}{0}{v^{-1}}\]
	is equal to
	\[ \morph{1+rpu^{-1}}{ -rpu^{-1}rv^{-1} }{vpu^{-1}}{-vpu^{-1}rv^{-1}+1}. \]
	For this to be an element of $\Lambda(G)$, the diagonal entries must both be 1.  So $rpu^{-1}=0$ and $-vpu^{-1}rv^{-1}=0$.  Since $u,v$ are isomorphisms, this simplifies to $rp=0$ and $pu^{-1}r=0$.  Note that $rp=0$ implies that $-rpu^{-1}rv^{-1}=0$, which is the other condition for $\Lambda(G)$ to be normal.  Now $rp=0$ for all choices of $r,p$ $\iff$ $Z(G)\leq G'$.  In other words, $G$ is necessarily a stem group, and therefore also purely non-abelian.  So we may suppose that $G$ is a stem group.  For any $g\in G$ and $z\in Z(G)$ the coefficient of $e_z$ in $r(g)$ is therefore 1.  Since $Z(G)\Char G$ and $u^*\in\Aut(G)$, we conclude that $pu^{-1}v=0$.
	
	This completes the proof.
\end{proof}

We can improve the situation by enlarging our consideration to $\Lambda_c(G)$.
\begin{lem}\label{lem:lc-bicrossed}
	If $G$ is purely non-abelian, then $\Lambda_c(G)$ and $\BCG$ give an exact factorization of $\Aut(\D(G))$.
\end{lem}
\begin{proof}
Suppose $G$ is purely non-abelian and $\psi=\ds{\morph{u}{r}{p}{v}}\in\Aut(\D(G))$.  We can therefore write
	\begin{align*}
		\morph{u}{r}{p}{v} &= \morph{u}{0}{p}{-pu^{-1}r+v} \morph{1}{u^{-1}r}{0}{1}\\
		&= \morph{1}{0}{pu^{-1}}{1} \morph{u}{0}{0}{-pu^{-1}r+v} \morph{1}{u^{-1}r}{0}{1}.
	\end{align*}
	Now, the matrix in the middle may be written as the product
	\[\morph{u(-pu^{-1}r+v)^*}{0}{0}{1} \morph{((-pu^{-1}r+v)^{-1})^*}{0}{0}{-pu^{-1}r+v}.\]
	We conclude that $\psi= ST$ with $S\in\Lambda_c(G)$ and $T\in\BCG$ where
	\begin{align*}
		S &= \morph{1}{0}{pu^{-1}}{1} \morph{u (-pu^{-1}r+v)^*}{0}{0}{1},\\
		T &= \morph{((-pu^{-1}r+v)^{-1})^*}{0}{0}{-pu^{-1}r+v} \morph{1}{u^{-1}r}{0}{1}.
	\end{align*}	
	Therefore $\Aut(\D(G)) = \Lambda_c(G)(\BCG)$, and clearly \[\Lambda_c(G)\cap\BCG =1.\]  So by definition, $\Lambda_c(G)$ and $\BCG$ give an exact factorization.
\end{proof}
We subsequently have the following variation of Theorem \ref{thm:stem-split1}.
\begin{thm}\label{thm:stem-split2}
	$\Lambda_c(G)$ is normal in $\Aut(\D(G))$ $\iff$ $G$ is a stem group.  Therefore, $\Aut(\D(G))\cong\Lambda_c(G)\rtimes(\BCG)$ $\iff$ $G$ is a stem group.
\end{thm} 
\begin{proof}
 We consider the conjugation action of $\BCG$ on $\Lambda_c(G)$.  We can compute that the conjugate
	\[\morph{(v^{-1})^*}{r}{0}{v} \morph{u}{0}{p}{1} \morph{v^*}{-v^*rv^{-1}}{0}{v^{-1}}\]
is equal to
	\[ \morph{-(v^{-1})^*uv^*+rpv^*}{-(v^{-1})^*uv^* rv^{-1}-rpv^*rv^{-1}+rv^{-1}}{vpv^*}{-vpv^*rv^{-1}+1}.\]
	
	For this to be an element of $\Lambda_c(G)$, we must have the upper right entry is $0$, and the lower right is $1$.  The last condition is equivalent to $-vpv^*rv^{-1}=0$, which is equivalent to $pv^*r = 0$.  This must hold for all choices of $p,v,r$ for normality to hold.  Thus we have $pr=0$ for all $p,r$ in particular, and as before this implies that $Z(G)\leq G'$.  Thus for normality to hold we must have that $G$ is a stem group, and so in particular purely non-abelian.  
	
	So we may suppose, then, that $G$ is a stem group, and therefore $pv^*r=0$ for all choices of $p,v,r$.  We need only show that the upper right entry, the bicharacter in the conjugate, is trivial.
	
	Since $G$ is a stem group, this bicharacter being trivial is equivalent to $(-(v^{-1})^*uv^{*}+1)r = 0$.  Since $u^*\in\Aut_c(G)$, it follows by normality that $vu^*v^{-1}\in\Aut_c(G)$ as well.  Therefore $(-(v^{-1})^*u v^*+1)$ is the dual of a group homomorphism $G\to Z(G)$.  Since $G$ is a stem group, it therefore follows that $(-(v^{-1})^*uv^{*}+1)r = 0$, as desired.
	
	This completes the proof.
\end{proof}

\begin{example}\label{ex:dihedrals3}
	Consider $G=D_{2n}$ with $4\mid n$.  Then $\widehat{G}\cong\BZ_2^2$, and subsequently $\BCh{G}\cong\BZ_2^4$.  Furthermore, $Z(G)\cong\BZ_2$, implying $\Lambda(G)\cong\BZ_2$, and it is routine to check that $\Aut_c(G)\cong\BZ_2^2$.  In particular, $\Lambda_c(G)\cong\BZ_2^3$ and $\SpAutc(G)\cong \BZ_2^2\rtimes \operatorname{Hol}(\BZ_n)$.  Since $G$ is a  stem group, the previous two theorems give us the isomorphisms 
\begin{align*}
	\Aut(\D(D_{2n}))&\cong\BZ_2^3\rtimes (\operatorname{Hol}(\BZ_n)\ltimes \BZ_2^4)\\
	&\cong \BZ_2\times ((\BZ_2^2\rtimes\operatorname{Hol}(\BZ_n))\ltimes \BZ_2^4).
\end{align*}
See also Examples \ref{ex:dihedrals1} and \ref{ex:dihedrals2}.  Note that in this case we have \[|\Aut(\D(D_{2n}))| = 2^7 n\phi(n),\] where $\phi$ is the Euler totient function.  For $n=4$ we see that $\Aut(\D(D_8))$ has order $2^{10}=1024$.
\end{example}
\begin{example}
	Let $Q$ be the quaternion group of order 8.  The following are all well-known:
	\begin{align*}
		Q'=Z(Q)&\cong\BZ_2,\\
		\widehat{Q}&\cong\BZ_2^2,\\
		\Aut(Q)&\cong S_4,\\
		\Aut_c(Q)&\cong \BZ_2^2.
	\end{align*}
In particular, $Q$ is a stem group.  It follows that $\Lambda(Q)\cong\BZ_2$ and $\BCh{Q}\cong \BZ_2^4$.  So the theorems provide us the isomorphisms
	\begin{align*}
		\Aut(\D(Q))&\cong \BZ_2\times ((\BZ_2^2\rtimes S_4)\ltimes\BZ_2^4)\\
		&\cong \BZ_2^3\rtimes (S_4\ltimes \BZ_2^4).
	\end{align*}
	In particular, $\Aut(\D(Q))$ has order $3072=2^{10}\cdot 3$.
\end{example}

Finally, we consider when $N$ has a complement.  Since $N$ is necessarily normal, this gives the  final case we consider of when $\Aut(\D(G))$ can be expressed as a semidirect product.
\begin{thm}\label{prop:N-split}
	Suppose $G$ is purely non-abelian. The kernel of the restriction $\Aut(\D(G))\to\Aut(\Gamma)$ has a complement in $\Aut(\D(G))$ $\iff$ $Z(G)\cap G'$ is a direct factor of $Z(G)$.
\end{thm}
\begin{proof}
	Since the kernel $N$ is normal and is contained in $\Lambda_c(G)$, by Lemma \ref{lem:lc-bicrossed} it suffices to prove that $N$ has a complement in $\Lambda_c(G)$.
	
	Let $u_c^*\in\Aut_{c'}$ and $p_c$ be such that $(p_c,u_c,1,\id)\in N$.  Then for any $(p,u,1,\id)\in\Lambda_c(G)$ we may write
	\begin{align*}
		\morph{u}{0}{1}{p}&=\morph{u_c}{0}{p_c}{1} \morph{u_c^{-1} u}{0}{-p_c u_c^{-1} u+p}{1}.
	\end{align*}
	We consider first the bottom left entries.  Since $Z(G)$ and $G'$ are both characteristic in $G$, it follows that $Z(G)\cap G'$ is characteristic in $G$.  Therefore we have that $(p_c u_c^{-1}u,\id,1,\id)\in N$.  So for $N$ to have a complement, we conclude that it is necessary for $\Hom(Z(G)\cap G',Z(G))$ to be a direct factor of $\End(Z(G))$.  Applying Fitting's Lemma to the canonical injection, we conclude that $\Hom(Z(G)\cap G',Z(G))$ is a direct factor of $\End(Z(G))$ if and only if $Z(G)\cap G'$ is a direct factor of $Z(G)$.
	
	So suppose that $Z(G)\cap G'$ is a direct factor of $Z(G)$, and write \[Z(G)=(Z(G)\cap G')\times C\] for some abelian group $C$.  Considering now the upper left entries, we see that for $N$ to have a complement we must have that the normal subgroup $\Aut_{c'}(G)$ has a complement in $\Aut_c(G)$.  Since $G$ is purely non-abelian, by \citep{AY65} there is a bijection $\Aut_c(G)\to\Hom(G,Z(G))$ given by $\phi\mapsto(g\mapsto \phi(g)g^{-1})$; letting $S$ denote the inversion map on the group, this is equivalent to $\phi\mapsto \phi*S$.  For any $\phi\in\Aut_c(G)$ let $z_{\phi}$ denote the image of $\phi$ under this bijection.  Thus $\phi=z_\phi*\id$.  We have the following identities:
	\begin{align}
		z_{\phi^{-1}} &= (S\circ z_\phi)\circ \phi^{-1}, \label{eq:autc-inv}\\
		z_{\phi\circ\psi} &= (z_\phi\circ\psi)*z_\psi.\label{eq:autc-comp}
	\end{align}
From this it is easy to see that the pre-image of $\Hom(G,C)$ forms a subgroup of $\Aut_c(G)$ which intersects $\Aut_{c'}(G)$ trivially.  We need only verify that $\Aut_c(G)$ is the product of these two subgroups.	By assumptions, we may uniquely write $z_\phi= z_2*z_1$ with $z_1\in\Hom(G,Z(G)\cap G')$ and $z_2\in\Hom(G,C)$.  Set $\phi_1= z_1*\id$ and $\phi_2=(z_2\circ\phi_1^{-1})*\id$.  Then $\phi=\phi_2\circ\phi_1$, and $z_{\phi_2}=z_2\circ\phi_1^{-1}\in\Hom(G,C)$.

This completes the proof.
\end{proof}
\begin{rem}
	It was essential that we knew $\Aut_{c'}(G)$ was normal.  In general, direct factors of $\Hom(G,Z(G))$ can be used to give exact factorizations of $\Aut_c(G)$, but in general it is not necessary that any of the subgroups so produced are normal.
\end{rem}

\begin{example}\label{ex:smallgroup-32-2}
	Define \[G=\cyc{x,y,z\ | \ x^4=y^4=z^2=1, \ [x,y]=z, \ [x,z]=[y,z]=1}.\]  This is a group of order 32, identified in \citet{GAP4} as SmallGroup(32,2).  It is easy to verify that $G$ is purely non-abelian, $Z(G)=\cyc{x^2,y^2,z}\cong\BZ_2^3$ and $G'=\cyc{z}\cong\BZ_2$.  Clearly $G'=Z(G)\cap G'$ is a direct factor of $Z(G)$.  Therefore the theorem applies, and the kernel $N$ is non-trivial and has a complement in $\Aut(\D(G))$.  Note that $G$ is not a stem group, so $\Lambda(G)$ and $\Lambda_c(G)$ are not normal.
\end{example}

Note that the structure, and even membership, of $\Aut(\D(D_{2n}))$ when $n\equiv 2 \bmod 4$ is not determined by our results, since then $D_{2n}\cong\BZ_2\times D_n$ is not purely non-abelian.  The completion of the general case, which is analogous to the case for groups \cite{BCM,Bi08}, is left for a future paper.  For now, we prove only the following special case.

\begin{prop}\label{sumsplits}
	Let $G=C\times H$, where $C,H$ have no common non-trivial direct factors.  Then \[\Aut(\D(G)) \cong \Aut(\D(C))\times\Aut(\D(H))\]
	if and only if $(|Z(\Gamma_C)|, |Z(\Gamma_H)|)=1$.
\end{prop}
\begin{proof}
	By Proposition \ref{prop:prodsub}, $\Aut(\D(C))\times\Aut(\D(H))$ is always a subgroup, so we need only determine when every automorphism is of this form.

	Note that $(|Z(C)||\widehat{C}|, |Z(H)||\widehat{H}|)=1$ is equivalent to the following three conditions all holding:
	\begin{enumerate}
		\item $(|\widehat{C}|,|\widehat{H}|)=1$;
		\item $(|Z(C)|,|Z(H)|)=1$;
		\item $(|\widehat{C}|,|Z(H)|)=(|\widehat{H}|,|Z(C)|)=1$.
	\end{enumerate}

	Now we observe that $\BCh{C\times H}\cong\BCh{C}\times\BCh{H}$ $\iff$ $(|\widehat{H}|,|\widehat{K}|)=1$, and that $\End(Z(H)\times Z(C))\cong\End(Z(H))\times\End(Z(C))$ $\iff$ $(|Z(H)|,|Z(C)|)=1$.
	
	Suppose that $\Aut(\D(G))=\Aut(\D(C))\times\Aut(\D(H))$.  Then under the imbedding of $\Aut(G)\hookrightarrow\Aut(\D(G))$, we conclude that $\Aut(G)=\Aut(C)\times\Aut(H)$.  Since $C,H$ have no common factors, by \citep{BCM} this holds $\iff$ both $\Hom(H,Z(C))$ and $\Hom(C,Z(H))$ are trivial.  By properties of the derived subgroup, this is equivalent to $\Hom(\widehat{H},Z(C))$ and $\Hom(\widehat{C},Z(H))$ both being trivial, which is in turn equivalent to $(|\widehat{H}|,|Z(C)|)=(|\widehat{C}|,|Z(H)|)=1$.
	
	Combining this with the observations on $\BCh{G}$ and $\End(Z(G))$, we conclude that $\Aut(\D(G))=\Aut(\D(C))\times\Aut(\D(H))$ implies that $(|Z(C)||\widehat{C}|, |Z(H)||\widehat{H}|)=1$.
	
	Now suppose that $(|Z(C)||\widehat{C}|, |Z(H)||\widehat{H}|)=1$.  It suffices to show that for $\morphquad\in\Aut(\D(G))$ that all four components split over $C,H$.  We have already observed that this is true for $p,r$, so we need only prove it is true for $u,v$.  By flipping, we need only prove it is true for $v$.
	
	Any $\psi\in\Aut(\D(G))$ restricts to an element of $\Aut(\Gamma_G)=\Aut(\Gamma_C\times\Gamma_H)$.  By the assumptions on $C,H$ and the orders, we conclude that $\Gamma_C$ and $\Gamma_H$ have no common direct factors.  Thus by \citep{BCM} the restriction of $\psi$ is an element of \[\left( \begin{array}{cc} \Aut(\Gamma_C)& \Hom(\Gamma_H,Z(\Gamma_C))\\ \Hom(\Gamma_C,Z(\Gamma_H))& \Aut(\Gamma_H) \end{array} \right).\]
	
	By the order assumptions we conclude that the $\Hom$ terms are all trivial.  On the other hand, in terms of $\morphquad$ we compute the off-diagonal terms to be
	\begin{align*}
		t(\chi_H,h)&=\pi_{\widehat{C}}(u(\chi_H,\varepsilon))\pi_C(v(h,1)),\\
		w(\chi_C,c)&= \pi_{\widehat{H}}(u(\varepsilon,\chi_C))\pi_H(v(1,c)),
	\end{align*}
	for all $\chi_H\in\widehat{H},\chi_C\in\widehat{C},h\in H,c\in C$.  Since these maps are trivial, we conclude that the projections involving $v$ are both trivial.  In particular, we conclude that $v\in\End(C)\times\End(H)$, as desired.
	
	This completes the proof.
\end{proof}
\begin{example}
	The proposition fails to apply to any of the following groups.
	\begin{enumerate}
		\item $A\times A_4$ for any abelian group $A$ with order divisible by 3, since $\widehat{A_4}\cong\BZ_3$.
		\item $A\times S_n$ for any abelian group $A$ of even order, since $\widehat{S_n}\cong\BZ_2$.  In particular, it fails to apply to $D_{2n}\cong \BZ_2\times D_n$ whenever $n\equiv 2\bmod 4$.
	\end{enumerate}
\end{example}
\begin{example}
The proposition applies to all of the following groups.
\begin{enumerate}
	\item $A\times G$, where $A$ is any abelian group and $G$ is any centerless perfect group.
	\item $A\times S_n$ and $A\times D_{2n}$, where $A$ is any abelian group of odd order.
	\item $A\times A_4$ where $A$ is any abelian group with order coprime to 3.
	\item $G\times H$ when $G$ and $H$ have coprime order.  This is the case if $G$ is a $p$-group and $H$ a $q$-group for distinct primes $p,q$, for example.  If $G$ or $H$ have abelian direct factors then computing the automorphisms of their doubles may itself be non-trivial, however.
\end{enumerate}
\end{example}

\section{Action on modules}\label{sec:modules}
In this section we describe the action of $\Aut(\D(G))$ on the category of representations $\Rep(\D(G))$.

Given algebras $A,B$ and an algebra map $f\colon A\to B$, we have the induced action $\Rep(B)\to\Rep(A)$ where $V\in\Rep(B)$ is sent to the $A$-module with the same vector space structure and action defined by $a.v=f(a).v$ for all $a\in A,v\in V$.  This is an equivalence of tensor categories when $A,B$ are Hopf algebras and $f$ is an isomorphism of Hopf algebras.  Modules are identified by their characters as usual, and we use the two interchangeably whenever convenient.

It is well known that the irreducible modules of $\D(G)$ are parametrized by (equivalence classes of) pairs $(s,\eta)$ where $s\in G$, and $\eta$ is an irreducible character of $C_G(s)$.  The pair depends only on $\operatorname{class}(s)$, the conjugacy class of $s$ in $G$, and the isomorphism class of $\eta$. See \citet{DPR} for further details.  More generally, $\eta$ can be any character of $C_G(s)$, and we obtain a module of $\D(G)$ in the same way: by inducing from $C_G(s)$ up to $G$.  As such, we use the notation $(s,\eta)$ to denote any such module, even when $\eta$ is not irreducible.  For $(s,\eta)$, the action of $e_g\# h$ is trivial whenever $g\not\in\operatorname{class}(s)$. We say that such a module is supported at $\operatorname{class}(s)$, or simply at $s$ when convenient.

The following description of the character of $(s,\chi)$ was proven in \citep{KSZ2}.
\begin{lem}\label{lem:ksz2}
	Consider a module of $\D(G)$ of the form $(s,\eta)$, with character $\chi$.  Let $\operatorname{class}(s)$ be the conjugacy class of $s$ in $G$.  For any $y\in\operatorname{class}(s)$ let $y'\in G$ be such that $s^{y'}=y$.  If $g\in\operatorname{class}(s)$ and $a=\gamma^{g'}$ for some $\gamma\in C_G(s)$, then $\chi(e_g\# a) = \eta(\gamma)$.  In all other cases, $\chi(e_g\# a)=0$.
\end{lem}

Let $\morphquad\in\Hom(\D(G),\D(H))$ have $B$ central.  Let $s\in H$ and let $\eta$ be any irreducible character of $C_H(s)$.  By Schur's Lemma we may write $\eta(z)=\mu(z)\eta(1)$ for all $z\in Z(H)$, where $\mu$ is some linear character of $Z(H)$.  Since $B$ is central, $B\leq C_H(s)$ for every $s\in H$.  The restriction of $\mu$ to $B$ is therefore also a linear character of $B$.  Define the following value for each such choice of $s,\eta$:
\begin{align}\label{eq:P-def}
	P(\eta) = \frac{1}{\dim\eta}\sum_{a\in A}\eta(p(e_a)) a.
\end{align}
Note that $P(\eta)$ is in $\mathbbm{k}A$ by construction.  Indeed, by the previous remarks and properties of $p$ it follows that there is a unique $a'\in A$ such that $\eta(p(e_a))=\dim\eta\delta_{a,a'}$, so that $P(\eta)\in A$ for all choices of $s,\eta$.

For the remainder of this section we will find it convenient to write the components in terms of their coefficients.  We use the original notation from \citep{ABM}:
\begin{align}
	p(e_g) &= \sum_{h\in H}\lambda(g,h)h,\\
	r(g) &= \sum_{h\in H}\omega(g,h)e_h.
\end{align}
Note that $\omega$ is a bilinear bicharacter.  In this notation, the values of $\psi=\morphquad\in\Hom(\D(G),\D(H))$ may be expressed as
\begin{align}
	\psi(e_g\# h) &= \sum_{y\in H, z\in B} \lambda(g u^*(y)^{-1},z)\omega(h,y)e_y\# z v(h),
\end{align}
whenever $u$ is a morphism of Hopf algebras.

Under mild assumptions, the induced map on representations will send a simple module to a (not necessarily simple) module supported on a single conjugacy class.
\begin{lem}\label{lem:mod-support}
	Let $\psi\in\Hom(\D(G),\D(H))$ have $A,B$ central.  Then the induced map $\Rep(\D(H))\to\Rep(\D(G))$ sends an irreducible module $(s,\eta)$ to a module supported at $P(\eta)u^*(s)$.
\end{lem}
\begin{proof}
	Let $\xi$ denote the character of the image module and $\chi$ the character of $(s,\eta)$.  Then by definition
	\begin{align}\label{eq:xi-base}
		\xi(e_g\# h) &= \sum_{z\in B}\sum_{y\in H} \lambda(g u^*(y)^{-1},z)\omega(h,y)\chi(e_y\# z v(h)).
	\end{align}
	By the preceding Lemma, the non-zero terms in the sum require $y\in \operatorname{class}(s)$ and $zv(h)=\gamma^{y'}$ for some $\gamma\in C_H(s)$; note that by centrality of $A$ this last condition is equivalent to saying $v(h)=\gamma^{y'}$ for some $\gamma\in C_H(s)$. Therefore we may write
	\begin{align}
		\xi(e_g\# h) &= \omega(h,s)\sum_{\substack{y\in \operatorname{class}(s)\\ v(h)=\gamma^{y'}}}\left( \sum_{z\in B} \lambda(g u^*(y)^{-1},z)\mu(z)\right)\eta(\gamma)\nonumber\\
		 &= \omega(h,s)\sum_{\substack{y\in \operatorname{class}(s)\\ v(h)=\gamma^{y'}}}\left(\frac{1}{\dim\eta}\eta(p(e_{gu^*(y)^{-1}}))\right)\eta(\gamma)\label{eq:xi-base2}
	\end{align}
As previously noted, the term in the parentheses is non-zero, and equal to one, precisely when $g u^*(y)^{-1} = P(\eta)$.  Combining, and using the assumption that $A\leq Z(G)$, we have $g\in P(\eta)u^*(\operatorname{class}(s))\subseteq \operatorname{class}(P(\eta)u^*(s))$.  Thus $\xi$ is supported at $P(\eta)u^*(s)$ as claimed.
\end{proof}
In equation \eqref{eq:xi-base2} we see that the summation over $y$ is counting the number of elements $y\in\operatorname{class}(s)$ such that $g u^*(y)^{-1}=P(\eta)$.  This number could be something other than 1 in general.  Indeed, it may happen that $P(\eta)u^*(\operatorname{class}(s))$ is a proper subset of $\operatorname{class}(P(\eta)u^*(s))$.

Clearly we have that $u^*(\operatorname{class}(s))=\operatorname{class}(u^*(s))$ if and only if for all $g\in\operatorname{class}(u^*(s))$ there exists a unique $y\in\operatorname{class}(s)$ with $g=u^*(y)$.  The addition of the central element $P(\eta)$ as above does not change this characterization.

\begin{lem}\label{lem:y-uniqueness-1}
	Let $G,H$ be finite groups and suppose \[\morphquad\in\Hom(\D(G),\D(H))\] is such that $G=A\Img(u^*)$ and $A\leq Z(G)$.  Then \[u^*(\operatorname{class}(s))=\operatorname{class}(u^*(s))\] for all $s\in H$.
\end{lem}
\begin{proof}
	For any $x\in G$ we may write $x=a u^*(h)$ for some $a\in A,h\in H$.  Then for any $s\in H$ we have
	\begin{align}
		x u^*(s) x^{-1} &= u^*(hsh^{-1}).
	\end{align}
	The result now follows.
\end{proof}
Note that the result is stronger than the observation that $\Img(u^*)\trianglelefteq G$ under the same assumptions.

It is easily checked $\morphquad\in\Hom(\D(G),\D(H))$ has $A,B$ central and $v\circ u^*$ normal if and only if $\morphquad$ is flippable. From this we get the description of the induced action on the representation categories.

\begin{thm}\label{thm:modaction}
	Let $\morphquad\in\Hom(\D(G),\D(H))$ be a flippable homomorphism such that $A\Img(u^*)=G$.  Then the induced map $\Rep(\D(H))\to\Rep(\D(G))$ is given by
	\begin{align}\label{eq:modaction}
		(s,\eta)\mapsto (P(\eta)\, u^*(s), r^*(s)*(\eta\circ v)).
	\end{align}
\end{thm}
\begin{proof}
	By the preceeding lemma, note that the assumptions guarantee $u^*(\operatorname{class}(s))=\operatorname{class}(u^*(s))$, and that $P(\eta)$ is defined and central.  So for $g\in\operatorname{class}(P(\eta)u^*(s))$ let $y_g\in\class(s)$ be such that $g=P(\eta)u^*(y_g)$.  By Lemma \ref{lem:mod-support} the image of $(s,\eta)$ is supported at $P(\eta)u^*(s)$.  We need only verify that the character of the image is equivalent to the one induced from $r^*(s)*(\eta\circ v)$.  
	
	First we must justify that $\eta\circ v$ is a well-defined character of \[C_G(P(\eta)u^*(s))=C_G(u^*(s)).\]  So suppose that $u^*(s)=x u^*(s) x^{-1}$ for some $x\in G$.  Applying $v$ to both sides, we have $vu^*(x) = v(x) vu^*(s) v(x^{-1})$.  By the assumption that $v\circ u^*$ is normal, we may write $vu^*(s)=s c$ for some $c\in C_H(\Img(v))$.  Thus we have $sc=vu^*(s)=v(x) s v(x)^{-1} c$, or equivalently that $v(x)\in C_H(s)$.  Thus $v$ maps $C_G(P(\eta)u^*(s))$ to $C_H(s)$ and $\eta\circ v$ is well-defined, as desired.
	
	Now let $\xi$ be the character of the image, and let $\beta$ be the character of $(P(\eta)\, u^*(s), r^*(s)*(\eta\circ v))$.  By Lemma \ref{lem:ksz2}, 
	\begin{align}
		\beta(e_g\# h)=\omega(h,s)\eta(v(t))=r^*(s)(h)\eta(v(t))
	\end{align}
whenever $g\in\operatorname{class}(P(\eta)u^*(s))$ and $h=t^{g'}$ for some $t \in C_G(P(\eta)u^*(s)) = C_G(u^*(s))$, and is zero otherwise.  On the other hand, by equation \eqref{eq:xi-base2} we have
	\begin{align}
		\xi(e_g\# h) = \omega(h,s) \eta(\gamma) = r^*(s)(h)\eta(\gamma)
	\end{align}
whenever $g\in\operatorname{class}(P(\eta)u^*(s))$ and $v(h)=\gamma^{y_{g}'}$ for some $\gamma\in C_H(s)$, and is zero otherwise.  So we need only show that the values $\eta(v(t))$ and $\eta(\gamma)$ coincide.  Indeed, we will show that $v(t)$ and $\gamma$ coincide.

So suppose then that we have $t^{g'}=h$, or equivalently $t=h^{(g')^{-1}}$.  Then $v(t)=v(h^{(g')^{-1}})$.  Since $A\Img(u^*)=G$ and $A\leq Z(G)$ we may find $x\in G$ and $a\in A$ such that $g'=a u^*(x)$.  Thus $h^{(g')^{-1}} = h^{u^*(x^{-1})}$.  Since $v\circ u^*$ is assumed normal, we then have $v(t) = v(h)^{x^{-1}}$.  Now set $y_0=s^{x}$.  Then $g u^*(y_0)^{-1} = g u^*(s^x)^{-1} = g (u^*(s)^{g'})^{-1}=P(\eta)$.  Since there is a unique $y\in\operatorname{class}(s)$ satisfying $gu^*(y)^{-1}=P(\eta)$, we conclude that $y_0$ is this value and that $y_{g}'$ can be taken equal to $x$.  This proves that $\eta(v(t))=\eta(\gamma)$, and so completes the proof.
\end{proof}
The theorem applies, in particular, to any element of $\Aut(\D(G))$.  In the case of an automorphism, the image module is necessarily simple.  Such a morphism need not be an isomorphism, however, as demonstrated by $\ds{\morph{0}{0}{p}{0}}\in\End(\D(A))$, where $A$ is an abelian group and $p$ is an isomorphism $\widehat{A}\cong A$.

We note, as an example, that the autoequivalences obtained from the subgroup $\BCG$ of $\Aut(\D(G))$ include the autoequivalences obtained by lifting the tensor autoequivalences of $\operatorname{Vec}_G$ through the center construction, as detailed in \citep[Example 6.11]{NR14}.  Indeed, the action in the Theorem is, after a few obvious definitions for notation, the matrix action
$$\morph{v^*}{r^*}{p^*}{u^*}(\eta,s)=(v^*\eta+r^*s,p^*\eta+u^*s),$$
where $(\eta,s)$ is the irreducible module $(s,\eta)$ with labels reversed. This is precisely the naively expected action of the flip on the module category.

\bibliographystyle{plainnat}
\bibliography{../references}
\end{document}